\documentclass[11pt]{article}
\usepackage[utf8]{inputenc}
\usepackage{amsmath}
\usepackage{amsfonts}
\usepackage{amsthm}
\usepackage{amssymb}
\usepackage{mathtools}
\usepackage[T1]{fontenc}
\usepackage{pst-node}
\usepackage{tikz-cd}
\usepackage{caption}
\usepackage{bbm}
\usepackage{comment}
\usepackage{array}   
\newcolumntype{L}{>{$}l<{$}}
\usepackage{multirow}
\usepackage{hyperref}

\usetikzlibrary{calc,backgrounds}

\pgfdeclarelayer{background}
\pgfdeclarelayer{foreground}
\pgfsetlayers{background,main,foreground}

\usepackage{geometry}
\theoremstyle{plain}
\newtheorem{thm}{Theorem}[section]
\newtheorem*{thm*}{Theorem}

\newtheorem{lem}[thm]{Lemma}
\newtheorem{prop}[thm]{Proposition}
\newtheorem{cor}[thm]{Corollary}

\newtheorem{ques}{Question}
\newtheorem*{claim}{Claim}

\newtheorem{mythm}{Theorem}

\theoremstyle{definition}
\newtheorem{defn}[thm]{Definition}

\theoremstyle{remark}
\newtheorem*{rem}{Remark}

\newcommand{\U}{\mathcal U}
\newcommand{\Ga}{\Gamma}
\newcommand{\Hi}{\mathcal H}
\newcommand{\act}{\curvearrowright}
\newcommand{\Op}{\mathcal O}
\newcommand{\eps}{\epsilon}

\newcommand{\re}{\mathbb{R}}

\newcommand{\na}{\mathbb{N}}

\newcommand{\Isom}{\mathrm{Isom}}

\newcommand{\Prob}{\mathrm{Prob}}

\renewcommand{\O}{\mathcal{O}}

\renewcommand{\a}{\alpha}

\newcommand{\G}{\Gamma}

\newcommand{\D}{\Delta}
\newcommand{\e}{\epsilon}

\renewcommand{\l}{\lambda}

\newcommand{\s}{\sigma}

\title{On growth of cocycles of isometric representations on $L^p$-spaces}

\author{Antonio López Neumann and Juan Paucar}

\date{\today}

\begin{document}

\maketitle

\begin{abstract}

    We study different notions of asymptotic growth for 1-cocycles of isometric representations on Banach spaces. One can see this as a way of quantifying the absence of fixed point properties on such spaces. Inspired by the work of Lafforgue, we show the following dichotomy: for a compactly generated group $G$, either all 1-cocycles of $G$ taking values in $L^p$-spaces are bounded (this is Property $FL^p$) or there exists a 1-cocycle of $G$ taking values in an $L^p$-space with relatively fast growth. We also obtain upper and lower bounds on the average growth of harmonic 1-cocycles with values in Banach spaces with convexity properties. As a consequence, we obtain bounds on the average growth of all 1-cocycles with values in $L^p$-spaces for groups with property $(T)$. Lastly, we show that for a compactly generated group $G$, the existence of a 1-cocycle with compression larger than $\sqrt{n}$ implies the Liouville property for a large family of probability measures on $G$.

    \vspace*{2mm} \noindent{2020 Mathematics Subject Classification:} 20J06,  22D05, 22D55, 46B08.


    \vspace*{2mm} \noindent{Keywords and Phrases:} locally compact groups, cocycles, uniformly convex spaces, property $FL^p$, ultraproducts, random walks.
\end{abstract}

\tableofcontents

\section{Introduction}

Let $G$ be a compactly generated locally compact group with compact generating set $S$ and associated word length $|\cdot |_S$. Let $\pi: G \to \mathcal{O}(E)$ be a continuous isometric representation of $G$ on some Banach space $E$. A \textit{$1$-cocycle} for the representation $\pi$ is a continuous map $b : G \to E$ that satisfies the so-called \textit{cocycle relation}:
\begin{equation*}
    b(gh) = b(g) + \pi(g) b(h),
\end{equation*}
for every $g, h \in G$. The space of 1-cocycles for $\pi$ is denoted by $Z^1(G, \pi)$.

Given a 1-cocycle $b$, we may look at how it "grows at infinity". This sentence can have different meanings. Here we investigate three different notions of growth for 1-cocycles. The first one is the \textit{diameter growth} or \textit{$L^\infty$-growth} of $b$, that is, the function given by:
\begin{equation*}
    n \mapsto \sup_{|g|_S \leq n} ||b(g)||_E.
\end{equation*}

We may also look at the \textit{equivariant compression of $b$}, this is the function defined by:
\begin{equation*}
    n \mapsto \rho_b(n) : = \inf_{|g|_S = n} ||b(g)||_E.
\end{equation*}

We also introduce a third notion of growth that will often interpolate between these two previous notions. Given a cohomologically adapted  probability measure $\mu$ on $G$ (see Definition \ref{Defn Cohomologically Adapted}) and $0<p<\infty$, we may look at the \textit{average growth of $b$ in $L^p(\mu)$-norm}, which we define by:
\begin{equation*}
     n \mapsto ||b||_{L^p(\mu^{*n})} := \Big( \int_G ||b(g)||_E^p \, \mathrm{d} \mu^{*n} (g) \Big)^{1/p}.
\end{equation*}

Most of the results in these notes are extensions of previously known results on the asymptotic behaviour of the growth of 1-cocycles of unitary representations from \cite{lafforgue-typeneg, cornulier-tessera-valette}, to isometric representations on (separable) $L^p$-spaces.

\paragraph{An $L^p$-analogue of a theorem by Lafforgue}

We are interested in (the lack of) fixed point properties on different types of Banach spaces. We start this discussion with the case of Hilbert spaces. We say that a locally compact group $G$ has \textit{Property $(FH)$} if every continuous isometric affine action of $G$ on a Hilbert space has a fixed point. Property $(FH)$ can be reformulated in terms of cocycles: a compactly generated locally compact group $G$ has Property $(FH)$ if and only if all 1-cocycles of $G$ on unitary representations are bounded. Hence looking at growth of 1-cocycles on unitary representations is only interesting for groups without property $(FH)$. It is natural to ask if there exist groups that are close to having property $(FH)$ in the sense that all 1-cocycles of unitary representations grow slowly. In \cite{lafforgue-typeneg}, Lafforgue showed that from the very moment a group does not have property $(T)$, it admits a 1-cocycle for some unitary representation whose diameter grows relatively fast (we recall that the Delorme-Guichardet theorem says that for $\s$-compact groups Property $(FH)$ and Property $(T)$ are equivalent \cite[Théorème V.1]{delorme-1-cohomologie-reps-unitaires} \cite[2.12.4]{bekka-delaHarpe-valette}).

\begin{thm}\cite[Théorème 3]{lafforgue-typeneg}\label{Intro: Lafforgue's theorem}
Let $G$ be a locally compact, compactly generated group with compact generating set $S$ and without property $(T)$. Then there exists a Hilbert space $\Hi$, a unitary representation $\pi : G \to \mathcal{U} (\Hi)$ and a nonzero 1-cocycle $b \in Z^1(G, \pi)$ such that for all $n\geq 1$:
\begin{equation*}
    \sup_{|g|_S \leq n} || b(g) ||_\Hi \geq \left(\frac{\sqrt{n}}{2} - 2 \right)  \sup_{|g|_S \leq 1} || b(g) ||_\Hi.
\end{equation*}
\end{thm}

Lafforgue's proof relies on a convergence argument for functions conditionally of negative type and a circumcenter argument.

Actions of groups on $L^p$-spaces have been studied since \cite{pansu89} and gained interest later with the works of \cite{fisher-margulis-inventiones}, \cite{BFGM} and \cite{yu-lp-spaces} (see also the survey in \cite[Chapter 6]{nowak-yu}).
The direct extension of property $(T)$ to an $L^p$-setting are the fixed point properties $FL^p$ for $0<p<\infty$, first introduced in \cite{BFGM}.

Throughout this article, \textit{$L^p$-space} means $L^p(\Omega, m)$ where $(\Omega, m)$ is some measured space and $0<p \leq \infty$. We will mostly focus on \textit{separable $L^p$-spaces}, which are known to coincide with \textit{standard $L^p$-spaces} (these are exactly $L^p([0, 1])$ with Lebesgue measure, $\ell^p (\na)$ with the counting measure, finite dimensional $\ell^p$ spaces and finite direct sums of these spaces). 

\begin{defn}
   A locally compact group $G$ is said to have \textit{property $FL^p$} if every continuous affine isometric action on any separable $L^p$-space has a fixed point. Equivalently, $G$ has $FL^p$ when all $1$-cocycles associated to isometric representations on separable $L^p$-spaces are bounded.
\end{defn}

Property $FL^p$ is equivalent to property $(T)$ for $p \in (0, 2]$ \cite{BFGM}. On the other hand, it is a stronger property for $p>2$, as any discrete hyperbolic group $\G$ admits affine actions without fixed points for $p> \mathrm{Confdim} (\partial \G)$ \cite{bourdon-pajot} and there exist hyperbolic groups with property $(T)$. Hence we can ask whether an analogue of Lafforgue's result holds for groups without property $FL^p$ for $p>2$. This is what we prove.

\begin{mythm}[Theorem \ref{LCLpproof1}]  \label{LCFLp1}
    Let $G$ be a locally compact, compactly generated group with compact generating set $S$. Suppose that $G$ does not have property $FL^p$ for some $2 \leq p < \infty$.
Then there exist a separable $L^p$-space $E$, a strongly continuous isometric representation $\pi\colon G\to \Op(E)$ and a nonzero cocycle $b\in Z^1(G,\pi)$ satisfying for every $n\geq 1$:
       \begin{equation*}
     \sup_{|g|_S\leq n} ||b(g)|| \geq \left(\frac 12 \left( \frac{1}{2^{p-1}}(n-1)+1 \right)^{\frac 1p} - 2\right) \sup_{|g|_S\leq 1} ||b(g)||.
 \end{equation*}
 
\end{mythm}

\begin{rem}
    Theorem \ref{LCFLp1} is mostly relevant when the group $G$ has property $(T)$. Indeed, \cite[3.1]{marrakchi-dlSalle} says that for every 1-cocycle $b$ of any unitary representation and for any $1<p< \infty$, there exists some separable $L^p$-space $E$, a continuous isometric action $\pi_p \colon G\to \mathcal \mathcal{O}(E)$ and a cocycle $b_p \in Z^1(G, \pi_p)$ such that for all $g \in G$ we have:
   \begin{equation*}
       || b_p(g)||_E = ||b(g)||_\Hi.
   \end{equation*}
    If $G$ does not have property $(T)$, plugging the cocycle given by Lafforgue's Theorem into this result yields for any $1<p< \infty$ a 1-cocycle $b_p$ on some separable $L^p$-space such that for every $n \geq 1$ we have:
   \begin{equation*}
     \sup_{|g|_S\leq n} || b_p(g)|| \geq \left(\frac{\sqrt{n}}{2} - 2 \right)  \sup_{|g|_S \leq 1} || b_p(g) ||.
   \end{equation*}
\end{rem}

In fact, our approach for Theorem \ref{LCFLp1} works in the more general setting of $p$-uniformly convex metric spaces (\cite[Definition 3.2]{naor-silberman}, see  Definition \ref{Def: p unif convex metric space} here). In this more general metric context, the corresponding fixed point properties are the following.
\begin{defn} Let $\mathcal{X}$ be a class of metric spaces.
 A locally compact group $G$ is said to have \textit{property $F \mathcal{X}$} if every continuous isometric action of $G$ on any metric space $X \in \mathcal{X}$ has a fixed point.
\end{defn}

Our extension of Theorem \ref{LCFLp1} to the metric setting will hold for classes $\mathcal{X}$ of metric spaces with the following properties.
We will say that $\mathcal{X}$ is \emph{stable under scaling} if for every $(E, d) \in \mathcal{X}$ and $t > 0$, we have $(E, t d) \in \mathcal{X}$. Fix a non-principal ultrafilter $\mathcal{U}$ of $\na$. We will say that $\mathcal{X}$ is \emph{stable under ultraproducts} if for every sequence $(E_i, d_i) \in \mathcal{X},  i \in \na$, we have that the ultraproduct metric space $(\prod_{\mathcal{U}} E_i, d_\mathcal{U}) \in \mathcal{X}$. These conditions will be enough to deal with actions of discrete groups. Examples of classes of $(p,c)$-uniformly convex metric spaces that are stable under scaling and under ultraproducts are the following: 
\begin{itemize}
    \item the class of Hilbert spaces ($p = 2$ and $c = 1$);
    \item for a fixed $1 < r < \infty$, the class of all (separable and non-separable) $L^r$-spaces   ($p = \max \{2, r \}$; and when $2 \leq r < \infty$ we have $c = \frac{1}{2^{r-1}}$);
    \item any class of $(p, c)$-uniformly convex Banach spaces which is stable under scaling and under ultraproducts;
    \item the class of CAT$(0)$ spaces ($p = 2$ and $c = 1$);
    \item the class of $\mathbb R$-trees ($p = 2$ and $c = 1$).
\end{itemize}

When considering strongly continuous representations of locally compact groups on Banach spaces, we will also need our classes of Banach spaces to be \textit{stable under smoothing} (see Definition \ref{defnUltraproductrep}).

The result we show is the following.

\begin{mythm}[Theorems \ref{discreteGeneralVersionproof} and \ref{LCLpproof1}]  \label{Intro: Fast cocycle growth for actions on metric spaces}
    Let $G$ be a locally compact, compactly generated group with compact generating set $S$, let $\mathcal{X}$ be a class of metric spaces for which $G$ does not have property $F \mathcal{X}$. Let $2 \leq p < \infty$ and $c>0$. Suppose one of the following conditions hold.
    \begin{enumerate}
        \item[(1)] The group $G$ is discrete and $\mathcal{X}$ is any subclass of $(p, c)$-uniformly convex metric spaces that is stable under scaling and ultraproducts. 
        \item[(2)]  $\mathcal{X}$ is any subclass of $(p, c)$-uniformly convex Banach spaces that is stable under scaling, ultraproducts and smoothing.
    \end{enumerate}
 Then there exists a space $(F,d) \in \mathcal{X}$, a continuous isometric action  $\s \colon G\to \Isom(F)$ and a basepoint $v \in F$ such that:
       \begin{equation*}
     \sup_{|g|_S\leq n} d(\s(g) v , v )  \geq \left(\frac 12(c(n-1)+1)^{\frac 1p} - 2\right) \sup_{|g|_S\leq 1} d(\s(g) v , v )
 \end{equation*}
 for all $n\geq 1$ and $\sup_{|g|_S\leq 1} d(\s(g) v , v )>0$.

\end{mythm}

The main difference between our techniques and those of \cite{lafforgue-typeneg}, is that instead of using properties of functions of conditionally negative type, we use an ultraproduct construction. This is the main idea to prove Theorem \ref{Intro: Fast cocycle growth for actions on metric spaces} $(1)$. To prove Theorem \ref{Intro: Fast cocycle growth for actions on metric spaces} $(2)$ we need to deal with the fact that the natural action on this ultraproduct may not be continuous: this is why we require the class $\mathcal{X}$ to be stable under smoothing.
Theorem \ref{LCFLp1} requires some additional care as an ultraproduct of separable $L^p$-spaces is a non-separable $L^p$-space: Proposition \ref{separable Lp space inside E_U}, proven in Appendix \ref{Section: Appendix}, takes care of this issue.

\paragraph{Harmonic cocycles and average growth}

We can recover Lafforgue's result in a more direct way using a theorem often attributed to Shalom: Property $(T)$ is equivalent to vanishing of the first reduced cohomology group for every unitary representation (cf. \cite[6.1]{shalom-rigidity}, see also \cite[3.8.B]{gromov2003random} and \cite[4.1.4]{korevaar-schoen}). For unitary representations, the first reduced cohomology group can be identified with the space of $\mu$-harmonic 1-cocycles, for any cohomologically adapted probability measure $\mu$ on $G$ (see Definition \ref{Defn Cohomologically Adapted}). This means that the absence of Property $(T)$ is equivalent to the existence of a nonzero $\mu$-harmonic cocycle of a unitary representation. Since a $\mu$-harmonic cocycle $b$ of a unitary representation satisfies:
\begin{equation*}
      ||b||_{L^2(\mu^{*n})} = \sqrt{n} \, ||b||_{L^2(\mu)}
\end{equation*}
(see e.g. \cite{erschler-ozawa}). Moreover, we can choose some cohomologically adapted measures such that $\mathrm{supp} (\mu) \subseteq S$, so that for all $n \geq 1$ we have:
\begin{equation*}
    \sup_{|g|_S \leq n} || b(g) ||_\Hi \geq ||b||_{L^2(\mu^{*n})},
\end{equation*} 
which implies Lafforgue's result with a (potentially) worse multiplicative constant.

We show that $\mu$-harmonic cocycles on $p$-uniformly convex spaces satisfy an estimate similar to the one obtained in Theorem \ref{LCFLp1}, but for averages in $L^p$-norm. We also obtain a complementary estimate for $q$-uniformly smooth spaces, though this inequality was already obtained in \cite{naor-peres-schramm-sheffield} and \cite{naor-peres} for discrete groups.

\begin{mythm} \label{Intro: Growth of harmonic cocycles}
Let $G$ be a locally compact, compactly generated group and let $\mu \in \Prob(G)$ be a cohomologically adapted probability measure on $G$. Let $\pi: G \to \mathcal{O}(E)$ be a continuous isometric representation of $G$ on a Banach space $E$ and let $b \in Z^1(G, \pi)$ be a $\mu$-harmonic 1-cocycle. \\
$\bullet$ If $E$ is a smooth $(p, c_E)$-uniformly convex Banach space, then for every $n \geq 1$ we have:
    \begin{equation*}
        ||b||_{L^p(\mu^{*n})} \geq  (c_E(n-1) + 1)^{\frac{1}{p}} ||b||_{L^p(\mu)}.
    \end{equation*}
$\bullet$ If $E$ is a $(q, d_E)$-uniformly smooth Banach space, then for every $n \geq 1$ we have:
    \begin{equation*}
        ||b||_{L^q(\mu^{*n})} \leq  (d_E(n-1) + 1)^{\frac{1}{q}} ||b||_{L^q(\mu)}.
    \end{equation*}
\end{mythm}

The inequalities obtained in Theorem \ref{Intro: Growth of harmonic cocycles} extend to all non-zero classes of 1-cocycles when the representation in question has spectral gap (up to a multiplicative constant depending on the cocycle). In particular, for groups with property $(T)$ (but without property $FL^p$ for some $2<p<\infty$), this applies to all isometric representations on separable $L^p$-spaces. 

\begin{mythm} \label{Intro: Not FLp and T implies fast growth on all cocycles}
    Let $G$ be a locally compact, compactly generated group with property $(T)$ but without property $FL^p$ for some $2 < p < \infty$ and $\mu \in \Prob(G)$ be a cohomologically adapted probability measure on $G$. Then for every continuous isometric representation $\pi: G \to \mathcal{O}(E)$ of $G$ on any separable $L^p$-space 
    $E$ such that $H_\mathrm{ct}^1(G, \pi) \neq 0$ and for every 1-cocycle $b \in Z^1(G, \pi) \setminus  B^1(G, \pi)$, there exist $C = C(b, p)$ such that for all $n \geq 1$ we have:
    \begin{equation*}
         ||b||_{L^p(\mu^{*n})} \geq C^{-1}  n^{1/p} \quad \text{ and }  \quad  ||b||_{L^2(\mu^{*n})} \leq C  n^{1/2}.
    \end{equation*}
        
\end{mythm}

A consequence of this fact is the absence in separable $L^p$-spaces 
($2 <p < \infty$) of a classical property of cocycles of unitary representations. Indeed, given a locally compact a-$T$-menable group $G$, there exist "arbitrarily slow proper cocycles", that is, for any proper function $f: G \to [1, + \infty )$ there exists a proper 1-cocycle $b$ on some unitary representation of $G$ such that $||b(g)|| \leq f(g)$ for $g \in G$ \cite[3.10]{cornulier-tessera-valette}. Theorem \ref{Intro: Not FLp and T implies fast growth on all cocycles} implies that given any locally compact group with property $(T)$, for any $2<p<\infty$, if $b$ is a 1-cocycle of some isometric representation on a separable $L^p$-space for which there is some $\a < 1/p$ such that $||b(g)|| \leq |g|^\a$  for all $g \in G$, then $b$ is bounded.

\paragraph{Drift and equivariant compression}

The existence of cocycles with fast compression growth imposes conditions on random walks.
The following result is originally due to Guentner and Kaminker \cite[Theorem 5.3]{guentner-kaminker} for discrete groups, and to Cornulier, Tessera and Valette for locally compact groups.

\begin{thm}\label{Intro: CTV thm}
    \cite[Theorem 1.2]{cornulier-tessera-valette} Let $G$ be a compactly generated locally compact group. Suppose that there exist a continuous unitary representation $\pi : G \to \mathcal{U}(\Hi)$ and a 1-cocycle $b \in Z^1(G, \pi)$ such that:
    \begin{equation*}
        \rho_b (n) = \inf_{|g|_S = n} ||b(g)||_{\Hi} \succ \sqrt{n},
    \end{equation*}
    that is, $\sqrt{n} / \rho_b(n)  \xrightarrow{n \to \infty}0$. Then $G$ is amenable.
\end{thm}

Naor and Peres showed a strengthening of this result in the discrete case (which was already hinted at in \cite{austin-naor-peres}).

\begin{thm}\label{Intro: Naor Paeres thm}
    \cite[Theorem 1.1]{naor-peres} Let $\G$ be a finitely generated group. Suppose that there exist some isometric representation $\pi : \G \to \mathcal{O}(E)$ on a $q$-uniformly smooth Banach space $E$, with $1 < q \leq 2$, and a 1-cocycle $b \in Z^1(\G, \pi)$ such that:
    \begin{equation*}
        \rho_b (n) = \inf_{|g|_S = n} ||b(g)||_{E} \succ n^{1/q}.
    \end{equation*}
    Then for every finitely supported probability measure $\mu \in \mathrm{Prob} (\G)$ with support generating $\G$, the drift $l(\mu)$ of the random walk driven by $\mu$ is $0$. This means that $(\G, \mu)$ has the Liouville property, that is, every bounded $\mu$-harmonic function on $\G$ is constant.
\end{thm}

The conclusion of Theorem \ref{Intro: Naor Paeres thm} is stronger than that of Theorem \ref{Intro: CTV thm}, as a discrete group $\G$ is amenable if and only if there exists $\mu \in \mathrm{Prob} (\G)$ such that $(\G, \mu)$ has the Liouville property \cite{kaimanovich-vershik}. Also, the hypothesis only requires a fast cocycle for some isometric representation on some $q$-uniformly smooth Banach space.

Our last theorem is an extension of Theorem \ref{Intro: Naor Paeres thm} to the non-discrete case. In the case of unitary representations, it can also be seen as a strengthening of the conclusion of Theorem \ref{Intro: CTV thm}.

\begin{mythm} \label{Intro: Compression implies drift=0}
    Let $G$ be a compactly generated locally compact group. Suppose that there exist a continuous isometric representation $\pi : G \to \mathcal{O}(E)$ on a $q$-uniformly smooth Banach space $E$ and a 1-cocycle $b \in Z^1(G, \pi)$ such that:
    \begin{equation*}
        \rho_b (n) = \inf_{|g|_S = n} ||b(g)||_{E} \succ n^{1/q}.
    \end{equation*}
    Then for every cohomologically adapted probability measure $\mu \in \mathrm{Prob} (G)$ with compact support $S$, the drift $l(\mu)$ of the random walk driven by $\mu$ is $0$. This means that $(G, \mu)$ has the Liouville property, that is, every continuous bounded $\mu$-harmonic function on $G$ is constant. 
\end{mythm}

The main ingredient of the proof of this Theorem (Proposition \ref{general inequality in q-unif smooth}) is the same as the one that Naor and Peres \cite[1.1]{naor-peres} used for discrete groups. We revisit their arguments in the locally compact setting, although in a more cohomological viewpoint.

\paragraph{Outline of the paper}

Section \ref{Section: Preliminaries} contains preliminaries used in the subsequent sections. It introduces convexity properties on metric spaces and Banach spaces, group cohomology and cohomologically adapted probability measures. Most of its content is standard and well known, except for Lemma \ref{Lplem}. Section \ref{Section: Cocycles on ultraproducts} introduces ultraproducts of isometric actions on Banach spaces and is devoted to the proofs of Theorems \ref{LCFLp1} and \ref{Intro: Fast cocycle growth for actions on metric spaces}. Section \ref{Section: Growth of harmonic cocycles} introduces harmonic cocycles before presenting proofs of Theorems \ref{Intro: Growth of harmonic cocycles} and \ref{Intro: Not FLp and T implies fast growth on all cocycles}. Section \ref{Section: Compression of cocycles and drift} recalls standard facts on random walks before giving a proof of Theorem \ref{Intro: Compression implies drift=0}. Section \ref{Section: Questions} contains questions for further research directions. Appendix \ref{Section: Appendix} contains a proof of the technical Lemma \ref{separable Lp space inside E_U} needed in the proof of Theorem \ref{LCFLp1}.

\paragraph{Acknowledgements}

The first author is funded by the National Science Center Grant Maestro-13 UMO-
2021/42/A/ST1/00306. He wants to thank Piotr Nowak for many discussions around convexity on Banach spaces and Eduardo Silva for discussions on random walks. The second author wishes to thank his advisor Romain Tessera for the suggestion of the problem and encouragement and Tomas Ibarlucía for discussions about $L^p$-spaces. We both thank Mikael de la Salle for helping us with separability issues on $L^p$-spaces, most notably, he communicated to us Proposition \ref{separable Lp space inside E_U} and its proof, given in Appendix \ref{Section: Appendix}. We thank the anonymous referee for suggesting us to present our results in a more general setting and for many comments that improved the exposition of the text.

\section{Preliminaries}\label{Section: Preliminaries}

This section introduces convexity properties on metric spaces, group cohomology and cohomologically adapted probability measures. We first introduce uniformly convex and uniformly smooth metric spaces. We then focus on the case of Banach spaces and introduce non-linear duality mappings.
We then introduce standard group cohomology vocabulary. We conclude this section by introducing cohomologically adapted probability measures.

\subsection{Uniformly convex and uniformly smooth metric spaces}

\begin{defn}
    We say that a complete metric space $(E,d)$ is \textit{uniquely geodesic} if for every two points $x, y \in E$, there exists a unique geodesic $\gamma$ between $x$ and $y$; that is, a path $\gamma \colon [0,1]\to E$ such that $d(\gamma(s),\gamma(t))=|s-t|d(x,y)$ for all $s,t \in [0,1]$. We will denote by $[x_0,x_1]$ this unique geodesic and by  $m_{xy}$ the middle point of the geodesic $[x,y]$.
\end{defn}

All Banach spaces are examples of complete uniquely geodesic metric spaces. We now start to narrow down the family of spaces we will work with, by introducing convexity properties.

\begin{defn}
Let $(E,d)$ be a uniquely geodesic complete metric space, we say that $E$ is a \textbf{uniformly convex metric space} if for each $\epsilon > 0$, there exists $\delta>0$ such that for all $x,y,z$ in  $E$ with $d(x,y) \geq \eps \max(d(x,z),d(y,z))$, we have that:
\begin{equation*}
    d(z,m_{xy}) \leq (1-\delta) \max(d(x,z),d(y,z)).
\end{equation*}
    
\end{defn}

These spaces are important because they are our minimal requirement to define a notion of center and radius of bounded subsets.

\begin{defn}
Given $(E,d)$ a uniformly convex metric space, for any bounded subset $A$ we define $c(A)$ and $\rho(A)$ as the center and the radius of the ball of smallest radius containing $A$.
\end{defn}

 We now introduce the two main classes of metric spaces we will be dealing with, that is, $p$-uniformly convex and $q$-uniformly smooth metric spaces.

\begin{defn} \label{Def: p unif convex metric space}
     We say that a uniquely geodesic complete metric space $(E,d)$ is $(p,c)$\textit{-uniformly convex} if for all points $x_0,x_1,z \in E$, we have that:
 \begin{equation*}
    d(z,x_t)^p\leq (1-t)d(z,x_0)^p + td(z,x_1)^p - ct(1-t)d(x_0,x_1)^p,
\end{equation*}
where $x_t$ is the linear parametrization of the segment $[x_0,x_1]$ for all $t \in [0,1]$. We will say that $(E,d)$ is \textit{$p$-uniformly convex} if it is $(p,c)$-uniformly convex for some $c>0$.
\end{defn}

\begin{defn} \label{Def: q unif smooth metric space}
     We say that a uniquely geodesic complete metric space $(E,d)$ is $(q,d)$\textit{-uniformly smooth} if for all points $x_0,x_1,z \in E$, we have that:
 \begin{equation*}
    d(z,x_t)^q\geq (1-t)d(z,x_0)^q + td(z,x_1)^q - dt(1-t)d(x_0,x_1)^q,
\end{equation*}
where $x_t$ is the linear parametrization of the segment $[x_0,x_1]$ for all $t \in [0,1]$. We will say that $(E,d)$ is \textit{$q$-uniformly smooth} if it is $(q,d)$-uniformly smooth for some $d>0$.
\end{defn}

The classical examples of such spaces are standard $L^p$-spaces. 

\begin{prop} \cite[Ch. 4 and 5]{chidume}
   Let $(\Omega, \mu)$ be some standard Borel space. \\ 
   $\bullet$ For $1<p \leq 2$, the space $L^p(\Omega, \mu)$ is $2$-uniformly convex and $p$-uniformly smooth. \\
   $\bullet$ For $2 \leq p < \infty$ the space $L^p(\Omega, \mu)$ is $p$-uniformly convex and $2$-uniformly smooth. 
\end{prop}


$CAT(0)$-cube complexes with $l^s$-metrics for $1<s<\infty$ as constructed in \cite{Nima} are examples of nonlinear $p$-uniformly convex and $q$-uniformly smooth metric spaces (where again $p = \max\{2,s\}$ and $q = \min \{2, s\}$).

We have the following lemma for $p$-uniformly convex metric spaces, that generalizes \cite[Lemma 4.7.1]{CornulierrelativepropT} from CAT($0$) spaces to $p$-uniformly convex metric spaces.

\begin{lem}
\label{Lplem}
Let $E$ be a $(p,c_E)$-uniformly convex metric space, $A$ any bounded subset of $E$ which is included in some other bounded subset $B$ of $E$. We then have the following inequality:
\begin{equation*}
    c_Ed(c(B),c(A))^p \leq (\rho(B)^p- \rho(A)^p). 
\end{equation*}
\end{lem}
\begin{proof}
Since $A\subset B\subset B(c(B),\rho(B))$, it is clear that $\rho(A)\leq \rho(B)$. Let's suppose that the conclusion of the lemma is false. In this case we can't have that $c(A)=c(B)$, so we can suppose that $c(A)$ is different from $c(B)$. Let's denote by $d$ the distance between $c(A)$ and $c(B)$ and by $c_t$ for $t$ in $[0,1]$ the point in $[c(A),c(B)]$ such that $d(c(A),c_t)=td$ and $d(c_t,c(B))= (1-t)d$. Moreover let's consider as well any point $a$ in $A$. We can see all these informations on the following figure:
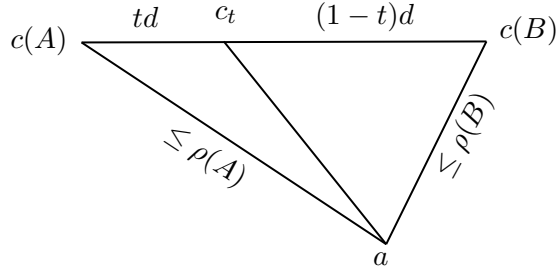
\begin{figure}[h]
    \centering
    \tikzset{every picture/.style={line width=0.75pt}} 

\begin{tikzpicture}[x=0.75pt,y=0.75pt,yscale=-1,xscale=1]

\draw    (244,79) -- (446,78.5) ;
\draw    (244,79) -- (396,180.5) ;
\draw    (446,78.5) -- (396,180.5) ;
\draw    (315,78.5) -- (396,180.5) ;

\draw (207,69.4) node [anchor=north west][inner sep=0.75pt]    {$c( A)$};
\draw (309,59.4) node [anchor=north west][inner sep=0.75pt]    {$c_{t}$};
\draw (388,183.4) node [anchor=north west][inner sep=0.75pt]    {$a$};
\draw (451,64.4) node [anchor=north west][inner sep=0.75pt]    {$c( B)$};
\draw (290.76,115.56) node [anchor=north west][inner sep=0.75pt]  [rotate=-33.57]  {$\leq \rho ({A}) \ $};
\draw (415.42,142.85) node [anchor=north west][inner sep=0.75pt]  [rotate=-298.49]  {$\leq \rho ({B})$};
\draw (268,59.4) node [anchor=north west][inner sep=0.75pt]    {$td$};
\draw (359,57.4) node [anchor=north west][inner sep=0.75pt]    {$( 1-t) d$};
\end{tikzpicture}
\end{figure}

We then have that for any point $a$ in $A$:        
\begin{equation*}
    d(a,c_{t})^{p}\leq (1-t) d(a,c(A))^{p}+td(a,c(B))^{p}-c_{E}t(1-t)d^{p}
\end{equation*}
We know as well that $a\in B(c(A),\rho(A))$ and since $A\subset B$, we have as well that $a\in B(c(B),\rho(B))$. This will give us the following bound:
\begin{equation*}
    d(a,c_t)^p\leq (1-t)\rho(A)^p + t \rho(B)^p -c_Et(1-t)d^p=h(t) 
\end{equation*}
Furthermore, if we suppose that $c_Ed^p > (\rho(B)^p-\rho(A)^p)$, we then get that the function 
\begin{equation*}
    h(t)= c_Ed^pt^2-(\rho(A)^p-\rho(B)^p+c_Ed^p)t +\rho(A^p)
\end{equation*} has its minimal value at $t_0= \dfrac{\rho(A)^p-\rho(B)^p+c_Ed^p}{2c_Ed^p}$ that belongs to $[0,1/2]$ and that $h(t_0)<\rho(A)^p$. We then have that for all $a$ in $A$:
\begin{equation*}
    d(a,c_{t_0})^p < \rho(A)^p.
\end{equation*}
Which contradicts the fact that $B(c(A),\rho(A))$ is the ball of minimal radius that contains $A$.
\end{proof}

\subsection{Duality mappings in smooth Banach spaces}

We are particularly interested in the case where our $p$-uniformly convex and $q$-uniformly smooth metric spaces are Banach spaces. An important characterization of these Banach spaces requires the introduction of duality mappings.
We follow the presentation in \cite{chidume}. 

Let $E$ be a real Banach space. We denote by $E^*$ its continuous dual with its norm $||f||_{E^*} : = \sup_{x \in E , ||x||_E = 1} |f(x)|$ for $f \in E^*$.

We denote by $\langle \cdot, \cdot \rangle: E \times E^{*} \to \re $ the natural pairing between $E$ and $E^*$ defined by $\langle x, f \rangle : = f(x)$, where $x \in E$ and $f \in E^*$.

\begin{defn}
    A Banach space $E$ is said to be \textit{smooth} if there exists a \textit{normalized duality mapping} $* : E \to E^*, x \mapsto x^*$ uniquely determined by the conditions $||x^*||_{E^*} = ||x||_E$ and $\langle x, x^* \rangle = ||x||_E^2$ for every $x \in E$.
\end{defn}

As the terminology indicates, a $q$-uniformly smooth Banach space ($1<q\leq 2$) is automatically smooth \cite[2.6]{chidume}.

\begin{rem}
    The map $* : E \to E^*$ is linear if and only if $E$ is a Hilbert space. In general it only satisfies $(\lambda x)^* = \l x^*$ for $x \in E$ \cite[3.7]{chidume}. In the case of an $L^p$-space, the normalized duality mapping is just the (normalized) Mazur map $L^p \to L^q, f \to M f$ defined by $M f (x) = \mathrm{sgn} (f(x)) |f(x)|^{p-1} ||f||_p^{2-p}$.
\end{rem}

For $1<r< \infty$, we write \begin{equation*}
    x^{*_r} := ||x||^{r-2} x^* \in E^*
\end{equation*} for $x \in E, x\neq 0$ and $x^{*_r} = 0$ for $x = 0$. 

The results in Section \ref{Section: Growth of harmonic cocycles} will follow from the following inequality.

\begin{prop}\label{p-unif convex inequality} \cite[4.17]{chidume}
    Let $E$ be a smooth $(p, c_E)$-uniformly convex Banach space. For all $x, y \in E$ we have:
    \begin{equation*}
        ||x+y||^p \geq ||x||^p + p \langle y , x^{*_p} \rangle +  c_E||y||^p.
    \end{equation*}
\end{prop}

Uniformly smooth spaces satisfy an inequality similar to that of uniformly convex spaces given in Proposition \ref{p-unif convex inequality}.

\begin{prop}\label{q-unif smooth inequality} \cite[5.8]{chidume}
  Let $E$ be a $(q, d_E)$-uniformly smooth Banach space. For all $x, y \in E$ we have:
    \begin{equation*}
        ||x+y||^q \leq ||x||^q + q \langle y , x^{*_q} \rangle +  d_E||y||^q.
    \end{equation*}
\end{prop}

\subsection{Isometric representations on Banach spaces and 1-cohomology}

We now introduce some group cohomology vocabulary. Standard references include \cite[2.2]{bekka-delaHarpe-valette}.

Let $G$ be a locally compact second countable group. We say that a representation $\pi: G \to \mathrm{GL}(E)$ of $G$ on a Banach space $E$ is \textit{continuous} (or \textit{strongly continuous}) if the map $G \times E \to E, (g, x) \mapsto \pi(g)x$ is continuous when $E$ carries the strong topology. It is \textit{isometric} if $\pi(g) \in \mathcal{O}(E)$ for all $g \in G$, where $\mathcal{O}(E)$ denotes the group of linear isometries of $E$. 

\begin{defn}
Let $\pi: G \to \mathcal{O}(E)$ be a continuous isometric representation of $G$ on a Banach space $E$.
   A \textit{$1$-cocycle} for the representation $\pi$ is a continuous map $b : G \to E$ satisfying the \textit{cocycle relation}:
\begin{equation*}
    b(gh) = b(g) + \pi(g) b(h),
\end{equation*}
for every $g, h \in G$. The space of 1-cocycles for $\pi$ is denoted by $Z^1(G, \pi)$. 
\end{defn}

Similarly, one defines \textit{1-coboundaries} as maps $b: G \to E$ for which there exists some $v\in E$ such that $b(g) = v - \pi(g) v$ for $g \in G$ and denote by $B^1(G, \pi)$ the space of all 1-coboundaries associated to $\pi$. The space $B^1(G,\pi)$ is a subspace of $Z^1(G, \pi)$ and we denote by $H^1_\mathrm{ct}(G, \pi) = Z^1(G, \pi) / B^1(G, \pi)$ the \textit{first continuous cohomology group} of $\pi$. 

All 1-coboundaries are bounded. If the Banach space $E$ is uniformly convex, then the converse also holds, that is, bounded 1-cocycles are exactly 1-coboundaries.

By considering an increasing family of compact subsets $(Q_n)_n$ of $G$, such that $\bigcup_n Q_n = G$, we see that $Z^1(G, \pi)$ is a Fréchet space for the family of seminorms $||b||_{Q_n} = \sup_{g \in Q_n} ||b(g)||$.

If the group $G$ is moreover compactly generated with compact generating set $S$, we endow $G$ with the word length $|\cdot |_S$ on $S$ defined by $|g|_S = \min \{ n \in \na , g \in S^n \}$ and the cocycle relation guarantees that the function $||b||_S = \sup_{g \in S} ||b(g)||$ is a norm on $Z^1(G, \pi)$, and the space $Z^1(G, \pi)$ endowed with $||\cdot||_S$ is a Banach space.


\subsection{Cohomologically adapted probability measures}

Let $G$ be a compactly generated locally compact group, $S$ be a compact generating set of $G$ and endow $G$ with the word length $|\cdot |_S$ on $S$ defined by
\begin{equation*}
    |g|_S = \min \{ n \in \na , g \in S^n \}
\end{equation*}
for $g \in G$. Let $\pi$ be a continuous isometric representation of $G$ on a Banach space $E$. For $b \in Z^1(G, \pi)$ define $||b||_S = \sup_{g \in S} ||b(g)||_E$. 

Let $\mu \in \Prob(G)$ be a probability measure on $G$. For $p >0$, we say that $\mu$ has \textit{finite $p$-moment} if:
\begin{equation*}
    \int_G |g|_S^p \, \mathrm{d} \mu (g) < \infty.
\end{equation*}

Under this condition, we define for the same value of $p>1$:
\begin{equation*}
    ||b||_{L^p(\mu)} :=  \Big( \int_G ||b(g)||_E^p \mathrm{d} \mu (g) \Big)^{\frac{1}{p}}.
\end{equation*}

The following definition encapsulates all the hypotheses we will further need.

\begin{defn} \cite[Definition 1]{bekka-harmonic2017} \label{Defn Cohomologically Adapted} Let $G$ be a compactly generated locally compact group with compact generating set $S$.
We say that a probability measure $\mu$ on $G$ is \textit{cohomologically adapted} if: \\
$\bullet$ $\mu$ is symmetric, \\
$\bullet$ $\mu$ is absolutely continuous with respect to the Haar measure $m_G$ on $G$, \\
$\bullet$ the support of $\mu$ contains $S$, \\
$\bullet$ $\inf_{x \in S} \frac{\mathrm{d} \mu}{\mathrm{d} m_G}(x) > 0$, \\
$\bullet$ $\mu$ has finite $p$-moment for every $1 \leq p < \infty$.
    
\end{defn}

The simplest example of such a probability measure is, for a unimodular group $G$, the restriction of the Haar measure $m_G$ to a compact symmetric generating subset. For a non-unimodular group $G$ with modular function $\D_G$, one considers the restriction of the measure $\Delta_G(g)^{-1/2} \mathrm{d} m_G(g)$ to a compact symmetric generating subset.

The following proposition states that given a cohomologically adapted probability and an isometric representation on a separable Banach space, the norms on the space of $1$-cocycles we have defined so far are equivalent. This result and the separability assumptions are not required in our further results, as these do not need to exchange these norms.

\begin{prop} Let $\mu$ be a cohomologically adapted probability measure on $G$ and assume that $E$ is a separable Banach space. Then for every $p \geq 1$, the function $||\cdot||_{L^p(\mu)}$ is a norm on $Z^1(G, \pi)$, so that $(Z^1(G, \pi), ||\cdot||_{L^p(\mu)})$ is a Banach space and the norms $||\cdot||_S$ and $||\cdot||_{L^p(\mu)}$ are equivalent. 
\end{prop}

\begin{proof} It is classical to show that the space $L^p(G, \mu, E)$ of all measurable maps $b: G \to E$ satisfying $||b||_{L^p(\mu)}< \infty$ modulo the relation $||b||_{L^p(\mu)} = 0$ is a Banach space. The closed subspace of maps satisfying the cocycle relation for $\pi$ is hence also a Banach space, we now show that this space coincides with $Z^1(G, \pi)$ by showing that these elements are automatically continuous outside a null set. 

Let $b: G \to E$ be a measurable map such that $||b||_{L^p(\mu)}< \infty$ and satisfying the cocycle relation. For $\epsilon >0$, it is enough to show that the set $B = \{ g \in G, ||b(g)|| < \epsilon \}$ contains an open neighborhood of $e_G$ (as the cocycle relation propagates continuity over all $G$). Since $E$ is separable, choose a sequence $g_n \in G$ such that the set $\{ b(g_n), n \in \na \}$ is dense in $b(G)$.  The cocycle relation implies that $G = \bigcup_{n \in \na} g_n B$, and hence $m_G(B) > 0$. 
We may also assume that $m_G(B)< \infty$. Under this condition, define the continuous function $\phi = \mathbf{1}_B * \mathbf{1}_B$. Since $\phi(e_G) = m_G(B)$, the open set $U = \phi^{-1} (0, + \infty)$ is an open neighborhood of $e_G$ and $U \subset B^2$.
Hence $b \in Z^1(G, \pi)$. 

For every $b \in Z^1(G, \pi)$ we have:
    \begin{equation*}
        ||b||_{L^p(\mu)} =  \Big( \int_G ||b(g)||_E^p \, \mathrm{d} \mu (g) \Big)^{\frac{1}{p}} \leq \Big( \int_G |g|_S^p \, \mathrm{d} \mu (g) \Big)^{\frac{1}{p}} ||b||_S.
    \end{equation*}
This proves that $(Z^1(G, \pi), ||\cdot||_{L^p(\mu)})$ is a Banach space and that the identity map on $Z^1(G, \pi)$ is a bijective continuous linear map from $(Z^1(G, \pi), ||\cdot||_S)$ to $(Z^1(G, \pi), ||\cdot||_{L^p(\mu)})$. Hence the Banach isomorphism theorem tells us that the inverse map is also continuous.
\end{proof}

\section{Growth of cocycles in $L^\infty$-norm} \label{Section: Cocycles on ultraproducts}

The goal of this section is to prove Theorems \ref{LCFLp1} and \ref{Intro: Fast cocycle growth for actions on metric spaces}. We start by introducing ultrafilters and ultraproducts of representations. We recall tools from \cite{UltraproductActions} to deal with continuity issues for representations on ultraproducts. We conclude the section by proving Theorem \ref{Intro: Fast cocycle growth for actions on metric spaces} and then Theorem \ref{LCFLp1}.

\subsection{Ultrafilters and ultraproducts of representations}

An ultrafilter on a set $I$ is a set $\U$ of subsets of $I$ that is closed under taking supersets, and such that for every subset $A$ of $I$, $\U$ contains either $A$ or $I\setminus A$ (but not both). As is standard, the set of ultrafilters on $I$ is in natural bijection with the set of characters of $l^\infty(I)$: an ultrafilter is something that chooses, for every bounded family $(a_i)_{i\in I}$ of complex numbers a point in the closure of $\{a_i: i\in I\}$ in a way compatible with pointwise multiplication and addition. 

If $\U$ is an ultrafilter on a set $I$, we denote by $\displaystyle (a_i)_{i\in I}\mapsto \lim_\U a_i$ the associated character on $l^\infty(I)$. It is characterized by the fact that $A\in \U$ if and only if $\lim_\U 1_{i\in A}=1$.

For our purposes we will only need ultrafilters on $\mathbb N$. We will say that $\U$ is a principal ultrafilter if it is the filter of all subsets that contain a fixed element. By Zorn's lemma it follows that there exist non-principal ultrafilters on $\mathbb N$. From now on, $\mathcal U$ will denote a non-principal ultrafilter on $\mathbb N$.

Given a family of pointed metric spaces $(E_i, v_i, d_i)_{i\in \mathbb N}$ where $v_i \in E_i$, we define the ultraproduct metric space $(\prod_\mathcal{U} E_i, v_\mathcal{U},  d_{\mathcal{U}})$ as
\begin{equation*}
    \prod_\mathcal{U} E_i := \left\{ (x_i)_{i \in \na} \in \prod_{i \in \na} E_i, \text{ the sequence } d_i (v_i, x_i)  \text{ is bounded} \right\} \Bigg/ \sim.
\end{equation*}
where $(x_i)_{i \in \na} \sim (y_i)_{i \in \na}$ when $ \lim_\mathcal{U} d_i (x_i, y_i) = 0$, and for $x = [(x_i)_{i \in \na}], y = [(y_i)_{i \in \na}] \in \prod_\mathcal{U} E_i$ we define $ d_{\mathcal{U}} (x, y ) := \lim_\mathcal{U} d_i (x_i, y_i)$ and $v_\mathcal{U} = [(v_i)_{i \in \na}]\in \prod_\mathcal{U} E_i$. The most important case we will consider is the case of ultraproducts of Banach spaces. In this case, the result space does not depend on the choice of the basepoint sequence $v_\mathcal{U}$, as translations are isometries. Hence we can always assume $v_\mathcal{U} = [0]$.

If $\G$ is a discrete group with isometric actions $\s_i$ on the metric spaces $E_i$, it admits natural isometric actions on $(\prod_\mathcal{U} E_i, v_\mathcal{U},  d_{\mathcal{U}})$ for any $v_\mathcal{U} \in \prod_\mathcal{U} E_i $. However, if we consider a locally compact group $G$ and a family of strongly continuous isometric representations $\pi_i\colon G\to \Op\left(E_i\right)$ on Banach spaces $E_i$, the natural representation $G\to \Op(\prod_\U E_i)$ will not be necessarily strongly continuous. In order to deal with this issue, we recall the following notions and propositions from \cite{UltraproductActions}:

\begin{defn}\label{defnUltraproductrep}
    Let $G$ be a locally compact group together, $\pi_i\colon G\to \Op(E_i)$ a family of strongly continuous isometric representations on Banach spaces $E_i$. Denote by $\pi\colon G\to \Op(\prod_\U E_i)$ the natural group morphism induced on $\prod_\U E_i$ (not necessarily strongly continuous). For any $f$ in $C_c(G)$ we will consider $\pi_i(f)x=\int_G f(g)\pi_i(g)x\  dm_G(g)$ and define $\pi(f)$ by the following $\pi(f).(x_i)_i=(\pi_i(f).x_i)_i$. Define then $E_\U$ as the closure of the space spanned by:
    \begin{equation*}
        \left\{\pi(f)x \ \big| \ x\in \prod_\U E_i, f\in C_c(G) \right\}.
    \end{equation*}
    We will denote by $\pi_\U$ the restriction of $\pi$ to $E_\U$. It can be seen that this is a strongly continuous isometric representation.

    Given a class of Banach spaces $\mathcal{X}$, we say that $\mathcal{X}$ is \emph{stable under smoothing} if for every sequence of strongly continuous isometric representations $\pi_i\colon G\to \Op(E_i)$ on Banach spaces $E_i \in \mathcal{X}$, we have that $E_\mathcal{U} \in \mathcal{X}$.
\end{defn}

\begin{prop} \cite[Proposition 4.12]{UltraproductActions}
\label{propMikael}
    Let $(\sigma_i)_{i \in \mathbb N}$ be a family of affine isometric actions of $G$ on Banach spaces $E_i$, with linear part $\pi_i$ and translation part $b_i\colon G\to E_i$, that is, $\sigma_i(g)\xi= \pi_i(g)\xi +b_i(g)$ for all $\xi$ in $E_i$. Assume that: \\
    $(i)$ The cocycles $b_i$ are pointwise bounded; that is, for every $g \in G$ we have:
    \begin{equation*}
        \sup_{i\in \mathbb N} ||b_i(g)|| < \infty.
    \end{equation*}
    $(ii)$ The cocycles are equicontinuous at the identity of $G$; that is, for all $\epsilon>0$, there exists a neighbourhood of the identity $U\subset G$ such that:
    \begin{equation*}
        \sup_{i\in \mathbb N} \sup_{g\in U}||b_i(g)||< \epsilon.
    \end{equation*}
    We then have that there is a continuous affine isometric action of $G$ on $E_\mathcal U$ with linear part $\pi_\mathcal U$ and translation part $b(g)=(b_i(g))_\mathcal U$. Moreover we have that:
    \begin{equation*}
        \sup_{|g|\leq n} ||b_\U(g)||=\lim_\U \sup_{|g|\leq n}||b_i(g)||.
    \end{equation*}
    \end{prop}

\begin{rem}
    Using techniques from \cite{Jeroen}, one can bypass Proposition \ref{propMikael} and obtain results similar to Theorem \ref{Intro: Fast cocycle growth for actions on metric spaces} for actions of locally compact groups on more general classes of metric spaces (not Banach), but with restrictions of other nature. Namely, the space playing the role of $E_\mathcal{U}$ in \cite[Proposition 1.6]{Jeroen} is only a closed convex subset in an ultraproduct, hence to obtain similar results we have to require our classes of metric spaces to be stable under taking closed convex subsets.
\end{rem}

We now discuss the relevant properties of ultraproducts of $L^p$-spaces.
Given a family of $L^p$-spaces $(E_i)_{i\in \mathbb N}$, it is known that the ultraproduct $\prod_\U E_i$ is an $L^p$-space \cite[Theorem 3.3 (ii)]{heinrich}. Moreover, we have the following.

\begin{prop} \cite[Proposition 4.11]{UltraproductActions} \label{E_U is an Lp space}
Let $E_i$ be an $L^p$-space for every $i$. Then $E_\U$ is also an $L^p$-space. In other words, the class of all $L^p$-spaces is stable under smoothing.
\end{prop}

A technical issue we will encounter is that even if we start from a separable $L^p$-space $E$, the $L^p$-space $E_\U$ may not be separable anymore. Nevertheless, it is possible to find a slightly smaller separable $L^p$-space on which $G$ still acts. 

\begin{prop}\label{separable Lp space inside E_U}
    Let $G$ be a locally compact second countable group acting by continuous affine isometries on some $L^p$-space $E$, there exists some separable $L^p$-space $F$, contained as a closed subspace inside $E$ that is $G$-invariant.
\end{prop}

The proof of this result is given in Appendix \ref{Section: Appendix}.

\subsection{Ultraproducts of isometric actions}

Let $G$ be a locally compact group together with $S$ a relatively compact symmetric neighborhood of the identity generating $G$. We will first prove the analogue of \cite[ Lemma 2]{lafforgue-typeneg} for $p$-uniformly convex metric spaces.

\begin{lem}
Given $G$, a $(p,c_E)$-uniformly convex metric space $E$ together with an isometric continuous action of $G$ on $E$ that doesn't have fixed points, we have that for every $N$ and $\epsilon > 0$ there exists $v \in E$ such that:
\begin{equation*}
    \rho(S^n\cdot v) \geq \left(c_E(1-\epsilon)(n-1)+1 \right)^{1/p} \rho(S\cdot v)
\end{equation*}
for all $n$ in $\{1,2,\cdots,N\}$.
\label{LpLafforgue}
\end{lem}

\begin{proof}
For any $v \in E$ let's denote $\rho(S^n\cdot v)$,\ $c(S^n\cdot v)$ by $\rho_n(v),c_n(v)$ respectively. 
The idea of the proof will be to suppose the conclusion of the lemma to be false and arrive at the following claim:
\begin{claim}
For any $v \in E$, we can find a $v' \in E$ such that $\rho_1(v')\leq (1-\eps)^{1/p}\rho_1(v)$ and $d(v',v)\leq C\rho_1(v)$ for some $C$ that is independent of $v$.
\end{claim}

This claim will allow us to arrive to a contradiction, since if we start with $v_0 \in E$ we can apply the claim recursively and obtain a sequence $(v_n)_{n\geq 0}$ such that:
\begin{align*}
    \rho_1(v_n) &\leq (1-\epsilon)^{\frac np} \rho_1(v_0)\\
    d(v_n,v_{n+1})&\leq C \rho_1(v_n).
\end{align*}
For all $n\geq 0$, which then implies that $(v_n)_{n\geq 0}$ is a Cauchy sequence. Let's denote by $\overline v$ its limit, it is then clear that $\overline v$ is a fixed point of the action giving us a contradiction.
 \end{proof}
 
\begin{proof}[Proof of Claim]
Suppose that the conclusion of Lemma \ref{LpLafforgue} is false, that is, there exists some $N$ and some $\epsilon > 0$ such that for every $v\in E$, there exists some $m$ in  $\{2,\cdots,N\}$ such that:
\begin{equation*}
    \rho_m(v)^p < (c_E(1-\epsilon)(m-1)+1) \rho_1(v)^p
\end{equation*}
This implies that there exists some $n$ in $\{1,2,\cdots,m-1\}$ such that:
\begin{equation*}
    \rho_{n+1}(v)^p - \rho_n(v)^p \leq c_E(1-\epsilon) \rho_1(v)^p.
\end{equation*}
Since for all $g$ in $S$ we know that $gS^n\cdot v\subset S^{n+1}\cdot v$ by Lemma \ref{Lplem} we have that:
\begin{equation*}
    c_E d(g\cdot c_n(v),c_{n+1}(v))^p \leq \rho_{n+1}(v)^p-\rho_n(v)^p \leq c_E(1-\epsilon) \rho_1(v)^p.
\end{equation*}
This implies that $S\cdot c_n(v)\subset B(c_{n+1}(v),(1-\epsilon)^{\frac 1p}\rho_1(v))$ which in turn implies that $\rho_1(c_n(v))\leq (1-\eps)^{1/p}\rho_1(v)$. It can also be seen that $d(c_n(v),v)\leq \rho_n(v)\leq \rho_m(v)\leq (c_EN+1)^{1/p}\rho_1(v)$ which finishes the proof of our Claim.
\end{proof}
In a similar manner as \cite{lafforgue-typeneg}, let's define $\displaystyle a_n(v)=\sup_{|g|\leq n}d(g.v,v)$ and note that:
\begin{equation*}
    2\rho_n(v) \geq a_{2n}(v) \geq a_n(v)\geq \rho_n(v)
\end{equation*}
The first inequality coming from the fact that $B_n.v\subset B(c_n(v),\rho_n(v))$ and that the diameter of $B_n.v$ is $a_{2n}(v)$; while the last inequality comes from the fact that $S^n\cdot v\subset B(v,a_n(v))$.

This implies then the following:
\begin{cor}\label{corLpLafforgue}
    Under the same hypothesis as Lemma \ref{LpLafforgue}, for all $N\in \mathbb N$ and $\epsilon>0$ there exists $v\in E$ such that:
    \begin{equation*}
        a_n(v) \geq  \left(c_E(1-\epsilon)(n-1)+1 \right)^{1/p} \dfrac{a_2(v)}2
    \end{equation*}
    for all $n$ in $\{1,2,\ldots, N\}$.
\end{cor}

We will now prove Theorem \ref{Intro: Fast cocycle growth for actions on metric spaces}, first for actions of discrete groups on $p$-uniformly convex metric spaces, and then for strongly continuous actions of locally compact groups on $p$-uniformly convex Banach spaces.

\begin{thm}\label{discreteGeneralVersionproof}
    Let $\mathcal{X}$ be a class of $(p, c)$-uniformly convex metric spaces that is stable under scaling and ultraproducts.
    Let $\Ga$ be a finitely generated group with word length $| \cdot |$. Suppose that $\G$ does not have property $F\mathcal X$. Then there exists a metric space $(F,d_F)$ in $\mathcal X$, a continuous isometric action $\sigma\colon G\to \mathrm{Isom}(F)$ and a basepoint $v\in F$ such that:
\begin{equation*}
    \sup_{|g|\leq n} d_F(\sigma(g)v,v) \geq \frac 12 \left(c(n-1)+1\right)^{\frac 1p} \sup_{|g|\leq 1}d_F(\sigma(g)v,v)
\end{equation*}
for all $n\in \mathbb N$ and such that $\sup_{|g|\leq 1} d_F(\sigma(g)v,v)>0$.    
\end{thm}
\begin{proof}
    Start with some metric space $(E,d_E)\in \mathcal X$ and a continuous isometric action $\sigma\colon G\to \mathrm{Isom}(E)$ of $\Ga$ on $E$. 
Since $\Ga$ does not have property $F\mathcal X$, we may assume that the action $\s$ does not have fixed points. By Lemma \ref{LpLafforgue} and Corollary \ref{corLpLafforgue}, there exists a sequence $(v_i)_{i\in \mathbb N}$ in $E$ such that:
\begin{equation*}
 \sup_{|g|\leq n}d_E(\s(g) v_i,v_i)\geq \frac 12 \left(c\left(1-\frac 1i\right)(n-1)+1\right)^{\frac 1p}   \sup_{|g|\leq 1}{d_E(\s(g) v_i,v_i)}.
\end{equation*}
for all $n$ in $\{1,\ldots,i\}$. We can then rescale the metric $d_E$ to obtain metric spaces $(E,d_i)\in \mathcal X$ such that for all $i$:
\begin{equation*}
    \sup_{|g|_S\leq 1} d_{i}(\s(g)v_i,v_i) = 1  
\end{equation*}
and such that:
\begin{equation*}
        \sup_{|g|\leq n} d_i(\s(g)v_i,v_i) \geq \frac 12 \left(c\left(1-\frac 1i\right)(n-1)+1\right)^{\frac 1p} \sup_{|g|\leq 1}d_i(\s(g)v_i,v_i)
\end{equation*}
for all $n$ in $\{1,\ldots,i\}$.
Consider the ultraproduct pointed metric space $(F,d_F,v)=\prod_\U (E,d_i,v_i)\in \mathcal X$, together with its natural isometric $\Ga$-action $\s_\mathcal{U}$. We have: 
\begin{equation*}
    \lim_{\mathcal U}\sup_{|g|\leq n} d_i(\s(g)v_i,v_i) \geq \lim_\mathcal U\frac 12 \left(c\left(1-\frac 1i\right)(n-1)+1\right)^{\frac 1p}.
\end{equation*}
Since $\sup_{|g|\leq n} \lim_{\mathcal U} d_i(\s(g)v_i,v_i) = \lim_{\mathcal U}\sup_{|g|\leq n} d_i(\s(g)v_i,v_i)$, we obtain:
\begin{equation*}
     \sup_{|g|\leq n} d(\s_\mathcal{U}(g)v,v)\geq \frac 12 \left(c(n-1)+1\right)^{\frac 1p} \sup_{|g|\leq 1} d(\s_\mathcal{U}(g)v,v).
\end{equation*}

\end{proof}

In the locally compact setting we have a similar statement. The proof of the following Theorem is essentially the same as that of Theorem \ref{discreteGeneralVersionproof}. It is technically more involved as we need to convolute in order to guarantee the continuity of the ultraproduct actions.

\begin{thm}
    \label{LCLpproof1}
       Consider $\mathcal X$ a class of $(p,c)$-uniformly convex Banach spaces that is closed under scaling, ultraproducts and smoothing. Let $G$ be a locally compact, compactly generated group with word length $| \cdot |$. Suppose that $G$ does not have $F\mathcal X$. Then there exists $F$ in $\mathcal X$, together with a strongly continuous isometric representation $\pi\colon G\to \Op(F)$ and a nonzero cocycle $b\in Z^1(G,\pi)$ such that:
       \begin{equation*}
     \sup_{|g|_S\leq n} ||b(g)|| \geq \left(\frac 12(c(n-1)+1)^{\frac 1p} - 2\right) \sup_{|g|_S\leq 1} ||b(g)||
 \end{equation*}
 for all $n\geq 1$.

\end{thm}

\begin{proof}
   Consider a $(p,c)$-uniformly convex Banach space $E$ with a continuous isometric affine action of $G$ on $E$
    . Assume that this affine action does not have fixed points. By Lemma \ref{LpLafforgue} and Corollary \ref{corLpLafforgue}, there exists a sequence $(v_i)_{i\in \mathbb N}$ in $E$ such that:
\begin{equation*}
 \sup_{|g|\leq n}||g.v_i-v_i||\geq \frac 12 \left(c\left(1-\frac 1i\right)(n-1)+1\right)^{\frac 1p}   \sup_{|g|\leq 2}{||g.v_i-v_i||},
\end{equation*}
for all $n\leq i$. We can then renormalize the cocycles $g\mapsto g.v_i - v_i$, obtaining new actions $\sigma_i\colon G\act E$ and vectors $\overline v_i \in E$, with 
cocycles $\overline b_i(g) := \sigma_i(g)(\overline v_i)-\overline v_i$ such that:
\begin{equation*}
    \sup_{|g|\leq 2} ||\overline b_i(g)|| = 1. 
\end{equation*}
The cocycles $\overline{b}_i$ satisfy for all $n\leq i$:
\begin{equation*}
    \sup_{|g|\leq n} ||\overline b_i(g)|| \geq \frac 12 \left(c\left(1-\frac 1i\right)(n-1)+1\right)^{\frac 1p}  \sup_{|g|\leq 2} ||\overline b_i(g)|| .
\end{equation*}
Consider then a function $\chi \colon G\to \mathbb [0,\infty)$ of compact support included in $V$, a symmetric neighbourhood of the identity such that $V.V\subset S$ and $V.S.V\subset S^2$ and $\int_{G}\chi(g) dm_G(g)=1$.
Let's define then the following:
\begin{equation*}
    \tilde v_i := \int_G \chi(h) \sigma_i(h)(\overline v_i)dm_G(h).
\end{equation*}
Consider the $1$-cocycles given by $b_i(g): =\sigma_i(g)(\tilde v_i)-\tilde v_i$. Let's first prove that:
\begin{equation*}
    \sup_{|g|\leq 1} ||b_i(g)||\leq 1.
\end{equation*}
This being true since we have that for all $g\in S$:
\begin{align*}
    ||b_i(g)|| &=||\sigma_i(g)(\tilde v_i)-\tilde v_i|| \\
    &=\left\lVert\int_V \int_V \chi(h_1)\chi(h_2) [\sigma_i(gh_1)\overline v_i-\sigma_i(h_2)\overline v_i]dm_G(h_1)dm_G(h_2)\right\rVert\\
    &\leq \int_V\int_V \chi(h_1)\chi(h_2) \left\lVert \sigma_i(h_2^{-1}gh_1)\overline v_i-\overline v_i \right\rVert dm_G(h_1)dm_G(h_2)\\
    &\leq 1,
\end{align*}
where the last inequality was obtained from the fact that $V.S.V\subset S^2$.
Now we can prove condition $(i)$ in Proposition \ref{propMikael}, that is for any $g\in G$, we have that:
\begin{equation*}
    ||b_i(g)||\leq |g|_S\sup_{|g|\leq 1}||b_i(g)||\leq |g|_S.
\end{equation*}
Now, in order to see that these cocycles satisfy condition $(ii)$ of Proposition \ref{propMikael}, let's consider $\epsilon >0$ and see that:
\begin{align*}
    \lVert b_i(g)\rVert  &= \left\lVert \int_G \chi(h)\sigma_i(gh).\overline v_i \ dm_G(h)-\int_G \chi(h)\sigma_i(h).\overline v_i \  dm_G(h)\right\rVert \\
    &= \left\lVert  \int_G \chi(g^{-1}h)\sigma_i(h).\overline v_i \ dm_G(h)-  \int_G \chi(h)\sigma_i(h).\overline v_i \ dm_G(h)\right\rVert\\
    &= \left\lVert \int_G    \chi(g^{-1}h)[\sigma_i(h).\overline v_i- \sigma_i(g).\overline v_i] \   dm_G(h) \right. \\
    & \left. + \int_G \chi(h)[\sigma_i(g).\overline v_i-\sigma_i(h).\overline v_i ] \ dm_G(h)\right\rVert \\
    &=\left\lVert\int_G [\chi(g^{-1}h)-\chi(h)][\sigma_i(h).\overline v_i-\sigma_i(g).\overline v_i]  \ dm_G(h)\right\rVert\\
    &\leq   \int_G |\chi(g^{-1}h)-\chi(h)| |h^{-1}g|_S \ dm_G(h).
\end{align*}
Let's choose $U$ a neighbourhood of the identity in $G$ such that $U\subset S$ and such that, for all $g$ in $U$, we have:
\begin{equation*}
    \int_G |\chi(g^{-1}h)-\chi(h)|<\epsilon.
\end{equation*}
Let's fix then $g$ in $U$. Since any $h$ not in $gV\cup V$ we have $|\chi(g^{-1}h)-\chi(h)|=0$, and for any $h\in gV\cup V$ we have that $|h^{-1}g|_S\leq 2$; we have that:
\begin{equation*}
    ||b_i(g)||\leq 2|\chi(g^{-1}h)-\chi(h)|<2 \epsilon,
\end{equation*}
which proves condition $(ii)$ of Proposition \ref{propMikael}. Hence Proposition \ref{propMikael} implies that the formula $b(g) := (b_i(g))_\U$ defines a 1-cocycle of the natural strongly continuous isometric representation on $E_\mathcal{U}$. 

It is clear by construction that:
\begin{equation*}
    ||\tilde v_i-\overline v_i|| = \left\lVert \int_V \chi(h)[\sigma_i(h).\overline v_i- \overline v_i]\ dm_G(h)\right\rVert \leq \int_V \chi(h) \ dm_G(h) = 1,  
\end{equation*}
which then implies that:
\begin{equation*}
    ||b_i(g)-\overline b_i(g)|| \leq ||\sigma_i(g).\tilde v_i-\sigma_i(g).\overline v_i||+ ||\tilde v_i-\overline v_i||\leq 2.
\end{equation*}
We then have that, for all $n\leq i$:
\begin{equation*}
    \sup_{|g|\leq n} ||b_i(g)||\geq  \sup_{|g|\leq n} ||\overline b_i(g)||-2\geq \frac 12 \left(c\left(1-\frac 1i\right)(n-1)+1\right)^{\frac 1p} -2.
\end{equation*}
From this inequality we obtain:
\begin{equation*}
    \sup_{|g|\leq n} ||b(g)||\geq \lim_\U \left( \frac 12 \left(c\left(1-\frac 1i\right)(n-1)+1\right)^{\frac 1p} -2\right)=\frac 12 \left(c(n-1)+1\right)^{\frac 1p}-2.
\end{equation*}

We have thus shown that if we consider the space $F=E_\mathcal U$, the strongly continuous representation $\pi\colon G\to \mathcal O(F)$  and the cocycle $b : G \to E_\mathcal{U}$, then $b$ satisfies the inequality in the statement of our Theorem. 
\end{proof}

\begin{proof}[Proof of Theorem \ref{LCFLp1}]
    To obtain the desired theorem for separable $L^p$-spaces, first recall that the class $\mathcal{X}$ of all $L^p$-spaces is stable under scaling, ultraproducts \cite[Theorem 3.3 (ii)]{heinrich} and smoothing (Proposition \ref{E_U is an Lp space}). Hence, given a compactly generated locally compact group $G$ without property $FL^p$ for some $p >2$, so in particular without $F \mathcal{X}$, we can apply Theorem \ref{LCLpproof1} to obtain an (a priori non-separable) $L^p$-space $E$, a strongly continuous isometric representation $\pi$ of $G$ on $E$ and a 1-cocycle $b$ with the required properties.

    By Proposition \ref{separable Lp space inside E_U}, there exists a separable $L^p$-space $F$ contained in $E$, that is $\pi(G)$-invariant and such that $b (G) \subseteq F$. Hence the cocycle $b$, associated to the strongly continuous representation $\pi : G \to \mathcal{O}(F)$ satisfies the required conditions.
\end{proof}

\section{Growth of harmonic cocycles in $L^p$-norm}\label{Section: Growth of harmonic cocycles}

This section is devoted to the proof of Theorems \ref{Intro: Growth of harmonic cocycles} and \ref{Intro: Not FLp and T implies fast growth on all cocycles}. It first introduces general properties of harmonic cocycles before proving the main results.

We begin by describing the common setting for all of this section. Let $G$ be a compactly generated locally compact group with compact generating set $S$ and associated word length $| \cdot |_S$ and let and $\mu \in \Prob(G)$ be a cohomologically adapted probability measure on $G$.

\subsection{Harmonicity and spectral gap}

We first introduce $\mu$-harmonic cocycles. Recall that a $1$-cocycle $b : G \to E$ of a continuous isometric representation on a Banach space $E$ is always Bochner-measurable (they are continuous, and defined on a separable space) and the sublinearity of a $1$-cocycle implies that $b$ is Bochner-integrable for any measure on $G$ with finite first moment.

\begin{defn}
   Let $E$ be a Banach space, $\pi: G \to \mathcal{O}(E)$ be a continuous isometric representation of $G$ on $E$ and $b \in Z^1(G, \pi)$. We say that $b$ is \textit{$\mu$-harmonic} if the following Bochner integral vanishes:
    \begin{equation*}
       b(\mu) := \int_G b(g) \mathrm{d} \mu (g) = 0.
    \end{equation*}
\end{defn}

$\mu$-harmonic cocycles are of particular importance for the study of unitary representations. The main drawback to use this definition in Banach settings is that it is hard to show their existence or to reduce problems to them. This is due to the fact that the duality mapping $*: E \to E^*$ is not linear when $E$ is not a Hilbert space. One situation in which we know that they are abundant is when our representation has spectral gap.

\begin{defn}
    Let $E$ be a Banach space and $\pi: G \to \mathcal{O}(E)$ be a continuous isometric representation of $G$ on $E$. We say that $\pi$ has \textit{spectral gap} if there exists $\e > 0$ such that for every $v \in E$ we have:
    \begin{equation*}
        \sup_{g \in S} ||v - \pi(g) v||_E \geq \e ||v||_E.
    \end{equation*}
\end{defn}

The following result is obtained by combining \cite[1.1]{drutu-nowak} and  \cite[6.1]{marrakchi-dlSalle}. It also appeared in the Hilbert setting in \cite[4]{bekka-harmonic2017} and in $L^p$-settings in \cite[3.12]{gournay-liouville}. We write a proof for completion.

\begin{prop}\label{spectral gap implies harmonic representative}
    Let $E$ be a uniformly convex Banach space and $\pi: G \to \mathcal{O}(E)$ be a continuous isometric representation of $G$ on $E$. If $\pi$ has \textit{spectral gap} then every cohomology class in $H_{\mathrm{ct}}^1(G, \pi)$ contains a unique $\mu$-harmonic representative. More precisely, every $b \in Z^1(G, \pi)$ can be written as $b = b_\mathrm{harm} + b_0$ where:
    \begin{equation*}
        b_0 (g) = v_0 - \pi(g) v_0, \quad v_0 = \sum_{i = 0}^{\infty} \pi(A_\mu)^{i} b(\mu), \quad \pi(A_\mu) = \int_G \pi(g) \, \mathrm{d} \mu (g)
    \end{equation*}
    and $b_\mathrm{harm} \in Z^1(G, \pi)$ is $\mu$-harmonic.
\end{prop}

\begin{proof}
    Let $b \in Z^1(G,\pi)$. Since $E$ is uniformly convex and $\pi$ has spectral gap, \cite[1.1]{drutu-nowak} implies that the operator $\pi(A_\mu) = \int_G \pi(g) \, \mathrm{d} \mu (g): E \to E$ satisfies $||\pi(A_\mu)|| < 1$. Hence the series 
    \begin{equation*}
        v_0 =  \sum_{i = 0}^{\infty} \pi(A_\mu)^{i} b(\mu)
    \end{equation*}
   converges to an element in $E$ and satisfies:
    \begin{equation*}
        \pi(A_\mu) v_0 =  \sum_{i = 1}^{\infty} \pi(A_\mu)^{i} b(\mu).
    \end{equation*}
    This means that if we set $ b_0 (g):= v_0 - \pi(g) v_0$, we have that $b_0 \in B^1(G, \pi)$ and that:
    \begin{equation*}
        b_0 (\mu) = v_0 - \pi(A_\mu) v_0 = b(\mu).
    \end{equation*}
    Hence $b_\mathrm{harm} : = b - b_0$ is $\mu$-harmonic. This shows the existence of a $\mu$-harmonic representative in each cohomology class. We will now show uniqueness. Let $b \in B^1 (G, \pi)$ be $\mu$-harmonic. Write $b(g) = v - \pi(g) v$ for some $v \in E$. We have $0 = b(\mu) = v - \pi(A_\mu) v$. Hence $ ||v|| = ||\pi (A_\mu) v || \leq || \pi (A_\mu)|| \, ||v|| < ||v||  $ and so $v = 0$.

\end{proof}

The following theorems show that Property $(T)$ (which can be defined as spectral gap on all unitary representations) implies spectral gap for all continuous isometric actions on $L^p$-spaces. 

\begin{thm}\label{(T) implies spectral gap on Lp} \cite[Theorem A]{BFGM}
    Let $G$ be a locally compact group with property $(T)$. Then all continuous isometric representations of $G$ on separable $L^p$-spaces for $1<p<\infty$ have spectral gap.
\end{thm}

\subsection{Average growth of harmonic cocycles}

In this section we prove Theorem \ref{Intro: Growth of harmonic cocycles} as a concatenation of the two following propositions.

\begin{prop}\label{mu-harmonic inequality in p-unif convex}
    Let $E$ be a smooth and $(p, c_E)$-uniformly convex Banach space, $\pi: G \to \mathcal{O}(E)$ be a continuous isometric representation of $G$ on $E$ and $b \in Z^1(G, \pi)$. If $b$ is $\mu$-harmonic, then for every $n \geq 1$ we have:
     \begin{equation*}
        ||b||_{L^p(\mu^{*n})} \geq  (c_E(n-1) + 1)^{\frac{1}{p}} ||b||_{L^p(\mu)}.
    \end{equation*}
\end{prop}

\begin{proof}
    Let $b \in Z^1(G, \pi)$. For any two cohomologically adapted probability measures $\mu, \nu \in \mathrm{Prob}(G)$, the cocycle relation and the symmetry of the measure $\mu$ give that: 
 \begin{align*}
        ||b||_{L^p(\mu * \nu)}^p & = \int_G \int_G || b(gh) ||^p \mathrm{d}\mu(g) \mathrm{d}\nu(h) \\
        & = \int_G \int_G || \pi(g)^{-1} b(g) + b (h) ||^p \mathrm{d}\mu(g) \mathrm{d}\nu(h) \\
        & = \int_G \int_G || b (h) - b(g^{-1}) ||^p \mathrm{d}\mu(g) \mathrm{d}\nu(h) \\
        & = \int_G \int_G || b (h) - b(g) ||^p \mathrm{d}\mu(g) \mathrm{d}\nu(h).
    \end{align*}

Now Proposition \ref{p-unif convex inequality} implies that:
\begin{align*}
        ||b||_{L^p(\mu * \nu)}^p
        & \geq \int_G \int_G || b(h) ||^p + c_E ||b(g)||^p - p \langle b(g), b(h)^{*_p} \rangle \mathrm{d}\mu(g) \mathrm{d}\nu(h) \\
        & = c_E ||b||_{L^p(\mu)}^p + ||b||_{L^p(\nu)}^p - p \,  \langle b (\mu), \int_G b(h)^{*_p} \mathrm{d}\nu(h) \rangle.
    \end{align*}

Hence if $b$ is $\mu$-harmonic, we have \begin{equation*}
    ||b||_{L^p(\mu * \nu)}^p \geq  c_E ||b||_{L^p(\mu)}^p + ||b||_{L^p(\nu)}^p.
\end{equation*} 

Now for $n \geq 1$ we can choose $\nu = \mu^{*(n-1)}$ and after iterating this inequality $(n-1)$ times we obtain: 
\begin{equation*}
    ||b||_{L^p(\mu^{*n})}^p \geq  c_E ||b||_{L^p(\mu)}^p + ||b||_{L^p(\mu^{*(n-1)})}^p \geq (1 + c_E(n-1)) ||b||_{L^p(\mu)}^p.
\end{equation*}

\end{proof}

\begin{prop}\label{mu-harmonic inequality in q-unif smooth}
    Let $E$ be a $(q, d_E)$-uniformly smooth Banach space, $\pi: G \to \mathcal{O}(E)$ be a continuous isometric representation of $G$ on $E$ and $b \in Z^1(G, \pi)$. If $b$ is $\mu$-harmonic, then for every $n \geq 1$ we have:
    \begin{equation*}
        ||b||_{L^q(\mu^{*n})} \leq  (d_E(n-1) + 1)^{\frac{1}{q}} ||b||_{L^q(\mu)}.
    \end{equation*}
\end{prop}

\begin{proof}
    The proof is essentially the same as that of Proposition \ref{mu-harmonic inequality in p-unif convex}, but instead of using the inequality of Proposition \ref{p-unif convex inequality}, we use the inequality from Proposition \ref{q-unif smooth inequality}.
\end{proof}

\subsection{Growth of cocycles for groups with Property $(T)$ but without $FL^p$}

Here we deal with consequences for groups with property $(T)$ but without Property $FL^p$.
Recall that hyperbolic groups do not satisfy $FL^p$ for large $p$ \cite{bourdon-pajot}, hence hyperbolic groups with property $(T)$ are an example of a family of groups we want to deal with here. These include $\mathrm{Sp}(n,1)$ for $n \geq 2$, the group $F_4^{-20}$ \cite[3.3]{bekka-delaHarpe-valette}, their cocompact lattices and many random groups \cite{drutu-mackay}.

We first prove Theorem \ref{Intro: Not FLp and T implies fast growth on all cocycles}.

\begin{proof}[Proof of Theorem \ref{Intro: Not FLp and T implies fast growth on all cocycles}]

    Suppose that $G$ has property $(T)$ but does not have Property $FL^p$ for some $p>2$. Then combining Theorem \ref{(T) implies spectral gap on Lp} and Proposition \ref{spectral gap implies harmonic representative} we obtain that for every isometric representation $\pi$ on any separable $L^p$-space $E$ and for every $b \in Z^1(G, \pi)$ there exists a $\mu$-harmonic cocycle $b_h\in Z^1(G, \pi)$ such that $b_0 : = b - b_h \in B^1(G, \pi)$. Recall that separable $L^p$-spaces for $2<p<\infty$ are $p$-uniformly convex and smooth. Hence using Proposition \ref{mu-harmonic inequality in p-unif convex} we obtain that for $n \geq 1$:
    \begin{align*}
        ||b||_{L^p(\mu^{*n})} & \geq ||b_h||_{L^p(\mu^{*n})} - ||b_0 ||_{L^p(\mu^{*n})} \\
        & \geq  (1 + c_E(n-1))^{\frac{1}{p}} ||b_h||_{L^p(\mu)} - \sup_{g \in G} ||b_0(g)||.
    \end{align*}
    Similarly, separable $L^p$ spaces are $2$-uniformly smooth for $2<p< \infty$, so Proposition \ref{mu-harmonic inequality in q-unif smooth} yields:
    \begin{align*}
        ||b||_{L^2(\mu^{*n})} & \leq ||b_h||_{L^2(\mu^{*n})} + ||b_0 ||_{L^2(\mu^{*n})} \\
        & \leq  (1 + d_E(n-1))^{\frac{1}{2}} ||b_h||_{L^2(\mu)} + \sup_{g \in G} ||b_0(g)||.
    \end{align*}
\end{proof}

We conclude this section with some remarks dealing with the existence of (proper) slow cocycles on separable $L^p$-spaces. Cornulier, Tessera and Valette showed that for a-$T$-menable groups, there always exist proper 1-cocycles with arbitrarily slow growth \cite[3.10]{cornulier-tessera-valette}. In fact, combining this result with \cite[3.1]{marrakchi-dlSalle} one can extend this statement to $1<p<\infty$. More precisely:

\begin{prop}\label{Slow proper cocycles}
    Let $G$ be a locally compact a-$T$-menable group. For every $1<p<\infty$ and for every proper function  $f : G \to [1, + \infty)$, there exists some isometric representation $\pi: G \to \mathcal{O}(E)$ on some separable $L^p$-space $E$ and a proper 1-cocycle $b$ of $\pi$ such that $||b(g)||_E \leq f(g)$ for all $g \in G$.
\end{prop}

Nevertheless, Theorem \ref{Intro: Not FLp and T implies fast growth on all cocycles} says that all a-$FL^p$-menable groups (that is, groups admitting a proper action by affine isometries on some separable $L^p$-space) for some $p>2$ with property $(T)$ impose restrictions on the proper functions $f$ with the same property. Indeed, Theorem \ref{Intro: Not FLp and T implies fast growth on all cocycles} implies that: given a group $G$ with property $(T)$, for every function of the form $f(g) = |g| ^\a$ with $0< \a < \frac{1}{p}$, all 1-cocycles $b$ on any $L^p$-space satisfying $||b(g)|| \leq f(g)$ for all $g \in G$ are automatically bounded.

\begin{rem}
It is worth noting that for any compactly generated locally compact (a-$T$-menable) group $G$, if some proper function $f : G \to [1, + \infty)$ is already the norm of a 1-cocycle $b : G \to \Hi$ of some unitary representation, that is $f(g) = ||b(g)||_\Hi$ for $g \in G$, then any cocycle $b': G \to \Hi'$ of any unitary representation satisfying $||b'(g)||_{\Hi'} \leq ||b(g)||_\Hi^\a$ for all $g \in G$ and some $0 <\a <1$ is automatically an almost coboundary. Indeed, our assumption together with Proposition \ref{Growth of Convolution norms} give that for every $n \geq 1$:
\begin{equation*}
     ||b ' ||_{L^2(\mu^{*n})}^2 \leq   ||b||_{L^{2 \a}(\mu^{*n})}^{2\a} \leq  ||b||_{L^2(\mu^{*n})}^{2\a} \leq n^{\a}.
\end{equation*}
So $\lim_{n \to \infty} \frac{1}{n} ||b ' ||_{L^2(\mu^{*n})}^2 = 0$ and hence $b'$ is an almost coboundary by \cite[2.2]{erschler-ozawa}.

Similarly, for a locally compact group $G$ with property $(T)$, if one assumes that $f(g) = ||b(g)||$ for some 1-cocycle $b$ on some $L^p$-space for some $2<p<\infty$, then any 1-cocycle $b'$ on any separable $L^p$-space such that $||b'(g)|| \leq f(g)^{2 \a /p}$ for all $g \in G$ and some $0 < \a <1$ is automatically bounded. Indeed if such a cocycle $b'$ was unbounded then Theorem \ref{Intro: Not FLp and T implies fast growth on all cocycles} implies that there exist constants $C, D > 0$ such that for large $n$:
\begin{equation*}
    C n \leq ||b ' ||_{L^p(\mu^{*n})}^p \leq ||b||_{L^{2 \a}(\mu^{*n})}^{2\a} \leq ||b||_{L^2(\mu^{*n})}^{2\a} \leq D n^{\a}.
\end{equation*}
This is absurd unless $\a = 1$.
\end{rem}

\section{Equivariant compression of cocycles and drift} \label{Section: Compression of cocycles and drift}

This section is devoted to the proof of Theorem \ref{Intro: Compression implies drift=0}. 
We first show a general inequality for 1-cocycles of unitary representations via a direct computation. We then revisit \cite[Theorem 2.1]{naor-peres} in the locally compact setting, with small modifications to their original proof.
We recall general properties of the drift of a random walk and prove Theorem \ref{Intro: Compression implies drift=0}. 

\subsection{A general inequality for cocycles}

The inequality obtained in the following proposition is similar to \cite[2.2]{erschler-ozawa}. We will not use this in the proof of Theorem \ref{Intro: Compression implies drift=0} as Proposition \ref{general inequality in q-unif smooth} is more general, but we notice that this is enough to deduce \cite[1.2]{cornulier-tessera-valette}.

\begin{prop}\label{Growth of Convolution norms}
    Let $G$ be a compactly generated locally compact group, with compact generating set $S$ and let $\mu \in \Prob(G)$ be a cohomologically adapted measure. Let $(\pi, \Hi)$ be a unitary representation of $G$. Then for every cocycle $b \in Z^1(G, \pi)$ and $n \in \na$, we have:

    \begin{equation*}
        ||b||_{L^2(\mu^{*n})} \leq \sqrt{n} \, ||b||_{L^2(\mu)}.
    \end{equation*}
    Moreover, this is an equality if and only if $b$ is $\mu$-harmonic.
    
\end{prop}

\begin{proof}
Let $\mu$ and $\nu$ be two cohomologically adapted measures. Recall that we denote by $b(\mu): = \int_G b(g) \, \mathrm{d} \mu(g)$ and by $\pi(A_\mu) := \int_G \pi(g) \, \mathrm{d} \mu(g)$.
    A direct computation shows that we have:
    \begin{align*}
        ||b||_{L^2(\mu * \nu)}^2 & = \int_G \int_G || b(gh) ||^2 \mathrm{d}\mu(g) \mathrm{d}\nu(h) \\
        & = \int_G \int_G || \pi(g)^{-1} b(g) + b (h) ||^2 \mathrm{d}\mu(g) \mathrm{d}\nu(h) \\
        & = \int_G \int_G || b(g) ||^2 + ||b(h)||^2 + 2 \langle \pi(g)^{-1} b(g), b(h) \rangle \mathrm{d}\mu(g) \mathrm{d}\nu(h) \\
        & =  ||b||_{L^2(\mu)}^2 + ||b||_{L^2(\nu)}^2 - 2  \left\langle \int_G b(g^{-1})  \mathrm{d}\mu(g), \int_G b(h) \mathrm{d}\nu(h) \right\rangle  \\
        & = ||b||_{L^2(\mu)}^2 + ||b||_{L^2(\nu)}^2 - 2  \left\langle \int_G b(g)  \mathrm{d}\mu(g), \int_G b(h) \mathrm{d}\nu(h) \right\rangle \\
        & = ||b||_{L^2(\mu)}^2 + ||b||_{L^2(\nu)}^2 - 2  \langle b(\mu), b(\nu)  \rangle.
    \end{align*}

If we apply this $n$ times to $\mu^{*n} = \mu^{*(n-1)} * \mu $  we have:
    \begin{align*}
        ||b||_{L^2(\mu^{*n})}^2 & = ||b||_{L^2(\mu^{*(n-1)})}^2 + ||b||_{L^2(\mu)}^2 - 2  \langle b(\mu^{*(n-1)}), b(\mu) \rangle \\
        & = n ||b||_{L^2(\mu)}^2 - 2  \left\langle \sum_{k = 1}^{n-1} b(\mu^{*k}), b(\mu) \right\rangle.
    \end{align*}

The cocycle relation for probability measures gives:
\begin{equation*}
 b(\mu^{*k}) = \sum_{j = 0}^{k-1} \pi(A_\mu)^j b(\mu).
\end{equation*}

Moreover, the operator $\pi(A_\mu)$ has operator norm 1 ($\mu$ is a probability) and is self-adjoint ($\mu$ is symmetric). Hence for $n\geq 3$:
    \begin{align*}
        \left\langle \sum_{k = 1}^{n-1} b(\mu^{*k}), b(\mu) \right\rangle 
        & = \left\langle \sum_{1 \leq k \leq n-1} \sum_{0 \leq j < k}\pi(A_\mu)^j b(\mu), b(\mu) \right\rangle \\
        & = \left\langle  \sum_{j = 0}^{n-2} (n-1-j) \pi(A_\mu)^j b(\mu), b(\mu) \right\rangle \\
        & =  \sum_{2j \leq n-2} (n-1-2j) || \pi(A_\mu)^j b(\mu) ||^2 \\
        & + \sum_{2j+1 \leq n-2} (n-2-2j) \langle \pi(A_\mu)^j b(\mu), \pi(A_\mu)^{j+1} b(\mu) \rangle \\
        &  \geq \sum_{2j \leq n-2} (n-1-2j) || \pi(A_\mu)^j b(\mu) ||^2 \\
        & - \sum_{2j+1 \leq n-2} (n-2-2j) || \pi(A_\mu)^j b(\mu) ||^2  \\
        & \geq \sum_{2j+1 \leq n-2} || \pi(A_\mu)^j b(\mu) ||^2 \geq 0.
    \end{align*}
    
Plugging this into our previous equality shows that the inequality
 \begin{equation*}
     ||b||_{L^2(\mu^{*n})}^2 \leq n \, ||b||_{L^2(\mu)}^2
 \end{equation*}
 holds, and that there is equality if and only if $\langle \sum_{k = 1}^{n-1} b(\mu^{*k}), b(\mu) \rangle = 0$, and this happens if and only if $b(\mu) = 0$.
\end{proof}

Naor and Peres showed a similar bound for 1-cocycles of finitely generated groups with values in isometric representations on $q$-uniformly smooth Banach spaces, generalizing both bounds obtained here in Propositions \ref{mu-harmonic inequality in q-unif smooth} and \ref{Growth of Convolution norms}. Their result relies on a theorem by Pisier \cite[4.1 and 4.2]{naor-peres-schramm-sheffield}. The approach by Naor and Peres also works in the locally compact setting. Here we give their statement in this generality, but we chose to present their proof using Proposition \ref{q-unif smooth inequality} instead of Pisier's theorem.

\begin{prop}\label{general inequality in q-unif smooth} \cite[Theorem 2.1]{naor-peres} Let $G$ be a compactly generated locally compact group, and let $\mu \in \mathrm{Prob}(G)$ be a cohomologically adapted measure.
    Let $E$ be a $(q, d_E)$-uniformly smooth Banach space, $\pi: G \to \mathcal{O}(E)$ be a continuous isometric representation of $G$ on $E$ and $b \in Z^1(G, \pi)$. For every $n \geq 1$ we have:
    \begin{equation*}
        ||b||_{L^q(\mu^{*n})} \leq  (2(d_E(n-1) + 1)^{\frac{1}{q}} + 1)  ||b||_{L^q(\mu)}.
    \end{equation*}
\end{prop}

\begin{proof}
    Let $g_1, \ldots, g_n \in G$ and define:
    \begin{align*}
        M_n(g_1, \ldots, g_n) &= b(g_1)-b(\mu)+\ldots+\pi(g_1\ldots g_{n-1})(b(g_n)-b(\mu)) \\ 
        &= \sum_{j=1}^n \pi (g_1 \ldots g_{j-1}) (b(g_j) - b(\mu)), \\ 
        N_n(g_1, \ldots, g_n) & = b(g_n^{-1})-b(\mu)+\ldots+\pi(g_n^{-1} \ldots g_{2}^{-1})(b(g_1^{-1})-b(\mu)) \\
        & = \sum_{j=1}^n \pi (g_n^{-1} \ldots g_{j+1}^{-1}) (b(g_j^{-1}) - b(\mu)).
    \end{align*}
    Using the cocycle relation, one can show that 
    \begin{equation*}
        2 b(g_1 \ldots g_n) = M_n(g_1, \ldots, g_n) - \pi (g_1 \ldots g_n) N_n(g_1, \ldots, g_n) + b(\mu) - \pi(g_1 \ldots g_n) b(\mu).
    \end{equation*}
    By integrating this equality we obtain the following upper bound:
    \begin{equation*}
         ||2 b||_{L^q(\mu^{*n})} \leq ||M_n||_{L^q (\mu^{\otimes n})} +  ||N_n||_{L^q (\mu^{\otimes n})} + 2 ||b(\mu)||.
    \end{equation*}
    Since the measure $\mu$ is symmetric, we have $||N_n||_{L^q (\mu^{\otimes n})} = ||M_n||_{L^q (\mu^{\otimes n})}$ and hence the previous inequality simplifies to: 
    \begin{equation*}
        || b||_{L^q(\mu^{*n})} \leq ||M_n||_{L^q (\mu^{\otimes n})} + ||b(\mu)||.
    \end{equation*}
    It remains to estimate $||M_n||_{L^q (\mu^{\otimes n})}$. For this we first notice that $M_n(g_1, \ldots, g_n) = M_{n-1}(g_1, \ldots, g_{n-1}) + \pi (g_1 \ldots g_{n-1}) (b(g_n) - b(\mu))$ and use Proposition \ref{q-unif smooth inequality} 
    \begin{align*}
        ||M_n(g_1, \ldots, g_n)||^q &\leq    ||M_{n-1}(g_1, \ldots, g_{n-1})||^q + d_E ||b(g_n) - b(\mu)||^q \\ 
        & +  q \langle  \pi (g_1 \ldots g_{n-1})( b(g_n) - b(\mu)) , M_{n-1}(g_1, \ldots, g_{n-1})^{*_q} \rangle. 
    \end{align*}
    Integrating with respect to $g_n$ cancels the last term, hence we obtain:
    \begin{equation*}
        ||M_n||_{L^q (\mu^{\otimes n})}^q \leq    ||M_{n-1}||_{L^q (\mu^{\otimes (n-1)})}^q + d_E \int_G ||b(g_n) - b(\mu)||^q \mathrm{d} \mu(g_n).
    \end{equation*}
    After iterating this inequality, we obtain $||M_n||_{L^q (\mu^{\otimes n})}^q \leq (d_E (n-1) +1 ) ||M_1||_{L^q (\mu)}^q$. We conclude that
    \begin{equation*}
        || b||_{L^q(\mu^{*n})} \leq (d_E (n-1) +1 )^\frac{1}{q}||M_1||_{L^q (\mu)} + || b||_{L^q(\mu)} \leq  (2(d_E (n-1) +1 )^\frac{1}{q} + 1)|| b||_{L^q(\mu)}.
    \end{equation*}
\end{proof}

\subsection{Drift and equivariant compression}

For a compactly generated locally compact group $G$ and Haar measure $m_G$, let $\mu \in \Prob(G)$ be a cohomologically adapted probability measure on $G$ with compact support $S$. Denote by $| \cdot |_S$ the word length in $S$. Denote by $X_n$ the random walk on $G$ with transition probability $\mu$: 
 \begin{equation*}
     \omega = (\omega_1, \omega_2, \ldots, \omega_n, \ldots) \mapsto X_n(\omega) =  \omega_1 \ldots \omega_n.
 \end{equation*}
The real valued function $|X_n|_S$ is sublinear in $n$. Hence by Kingman's subadditive ergodic theorem the limits 
\begin{equation*}
     \lim_n \frac{|X_n|_S}{n} = \lim_n \frac{1}{n} \int_{G} | g |_S \,  \mathrm{d} \mu^{*n} (g) = \inf_n \frac{1}{n} \int_{G} | g |_S \,  \mathrm{d} \mu^{*n} (g)
\end{equation*}
exist and are constant almost everywhere \cite{karlsson-ledrappier-drift}. The \textit{drift} (or \textit{escape rate}) of the random walk is the almost everywhere value of the limit $l(\mu) = \lim_n \frac{|X_n|_S}{n} \geq 0$.

The following is the analogue of \cite[1.1]{austin-naor-peres}.

\begin{prop}\label{Compression and convolution norm}
    Suppose that there exists $\epsilon > 0$ and a function $f: \re_+ \to \re_+$ that is increasing at most at linear speed, such that $f(t) \xrightarrow[]{t \to + \infty} + \infty$ and such that for all $n \in \na$:
    \begin{equation*}
        \mathbb{P}(|X_n|_S \geq f(n)) \geq \epsilon.
    \end{equation*}
    Then for every continuous isometric representation $(\pi, E)$ of $G$ on a Banach space $E$ and every $b \in Z^1(G, \pi)$, for every $1 <p<\infty$ and $n \in \na$ we have:
    \begin{equation*}
         \rho_b(f(n )) \leq \epsilon^{-1/p} ||b||_{L^p(\mu^{*n})}.
        \end{equation*}
\end{prop}

\begin{proof}
    We have:
    \begin{align*}
        ||b||_{L^p(\mu^{*n})}^p &= \int_G ||b(g)||^p \mathrm{d} \mu^{*n} (g) \\
        & \geq \int_{ \{|g|_S \geq f(n) \} } ||b(g)||^p \mathrm{d} \mu^{*n} (g) \\
        & \geq \rho_b(f(n))^p \, \mathbb{P}(|X_n|_S \geq f(n)) \\
        & \geq \epsilon \rho_b(f(n))^p.
    \end{align*}
\end{proof}

\begin{prop}\label{Upper bound for compression}
    Suppose that $l(\mu) > 0$. \\
     Let $E$ be a $(q, d_E)$-uniformly smooth Banach space, $\pi: G \to \mathcal{O}(E)$ be a continuous isometric representation of $G$ on $E$. There exists a constant $C = C(\mu, q, d_E) > 0$ such that for every 1-cocycle $b \in Z^1(G, \pi)$ and for every $n \in \na$ we have:
    \begin{equation*}
        \rho_b(n) \leq C ||b||_{L^q(\mu)} n^\frac{1}{q}.
    \end{equation*}
    
\end{prop}

\begin{proof}
Since $\frac{|X_n|_S}{n}$ converges almost everywhere to $l(\mu) > 0$ when $n \to \infty$, this convergence is also in probability. Hence for every $\epsilon >0 $ we have:
\begin{equation*}
    \mathbb{P}\left( \left| \frac{|X_n|_S}{n} - l(\mu) \right| \geq \epsilon\right) \xrightarrow{n \to \infty} 0.
\end{equation*}
In particular, for $\epsilon = l(\mu) /2$ we have that:
\begin{equation*}
    \mathbb{P}\left(  \frac{|X_n|_S}{n} \geq \frac{l(\mu)}{2}\right) \xrightarrow{n \to \infty} 1.
\end{equation*}
Choose $N = N(\mu)$ large enough so that $\mathbb{P}\left( \frac{|X_n|_S}{n} \geq \frac{l(\mu)}{2}\right) \geq \frac{1}{4}$ for $n \geq N$. 

Let $b$ be a 1-cocycle of some continuous isometric representation of $G$ on a $q$-uniformly smooth Banach space $E$. Using Proposition \ref{Compression and convolution norm} and Proposition \ref{general inequality in q-unif smooth}, there exists a constant $D = D(q, d_E)>0$ such that for $n \geq N$ we have:
\begin{equation*}
    \rho_b\left( \frac{l(\mu) n}{ 2} \right) \leq 4^\frac{1}{q} ||b||_{L^q(\mu^{*n})} \leq  4^\frac{1}{q} D ||b||_{L^q(\mu)} n^\frac{1}{q}.
\end{equation*}
Hence for $n \geq N$:
\begin{equation*}
    \rho_b(n )  \leq  \Big(\frac{8}{l(\mu)}\Big)^\frac{1}{q} D ||b||_{L^q(\mu)}n^\frac{1}{q}.
\end{equation*}

\end{proof}

\subsection{Large equivariant compression}

We will now prove Theorem \ref{Intro: Compression implies drift=0}. 

\begin{proof}[Proof of Theorem \ref{Intro: Compression implies drift=0}]
    Let $G$ be a compactly generated locally compact group and suppose that
    there exist a continuous isometric representation $\pi : G \to \mathcal{O}(E)$ on a $q$-uniformly smooth Banach space $E$ and a 1-cocycle $b \in Z^1(G, \pi)$ such that:
    \begin{equation*}
        \rho_b (n) = \inf_{|g|_S = n} ||b(g)||_{E} \succ n^{1/q}.
    \end{equation*}

     The contrapositive of Proposition \ref{Upper bound for compression} allows one to conclude that $l(\mu) = 0$. If $G$ is a discrete group, it is classical to show that $l(\mu) = 0$ implies that $(G, \mu)$ has the Liouville property. For example, this can be shown as a combination of \cite[1.1]{kaimanovich-vershik} and the computation in \cite[Section 4]{karlsson-ledrappier-drift}.
    
    If $G$ is a compactly generated locally compact group, one can define the asymptotic entropy $h(\mu)$ of the random walk driven by $\mu$ by using densities as in \cite[Section III]{derriennic-entropie}. One can reproduce the argument of \cite[Section 4]{karlsson-ledrappier-drift} in the locally compact case to show that the condition $l(\mu) = 0$ still implies that $h(\mu) = 0$. We obtain that $(G, \mu)$ has the Liouville property using \cite[Théorème p.255]{derriennic-entropie}.
\end{proof}

\section{Questions}\label{Section: Questions}


As in Proposition \ref{Slow proper cocycles}, we say that a locally compact compactly generated group $G$ \textit{has arbitrarily slow proper cocycles in $L^p$} (where $1<p<\infty$) if for every proper function $f: G \to [0, \infty)$, there exist some $L^p$-space $E$, a continuous isometric representation $\pi: G \to \mathcal{O}(E)$ and a proper cocycle $b \in Z^1 (G, \pi)$ such that $||b(g)|| \leq f(g)$.

We know that a-$T$-menable groups have arbitrarily slow proper cocycles in $L^p$ for every $1<p< \infty$ \cite[Proposition 3.1]{marrakchi-dlSalle}. We also know that property $(T)$ groups do not have arbitrarily slow proper cocycles in $L^p$ for every $1<p< \infty$ because of Theorem \ref{Intro: Not FLp and T implies fast growth on all cocycles}. 

\begin{ques}\label{Question: Slow cocycles}
Let $G$ be a locally compact compactly generated group. Suppose that $G$ is a-$FL^p$-menable (i.e. $G$ admits a proper cocycle on an $L^p$-space) for some $2<p<\infty$, not a-$T$-menable and does not have property $(T)$. Does $G$ have arbitrarily slow proper cocycles in $L^p$? 
\end{ques}
Examples of groups with conditions as in Question \ref{Question: Slow cocycles} include products of a-$T$-menable groups with property $(T)$ groups, such as $\mathrm{SO}(n, 1) \times \mathrm{Sp}(m,1)$ for $n, m \geq 2$.

Some of the techniques in Section \ref{Section: Growth of harmonic cocycles} work not only for $L^p$-spaces, but for the larger class of $p$-uniformly convex Banach spaces. An important ingredient in the proof of Theorem \ref{Intro: Not FLp and T implies fast growth on all cocycles} is that every cohomology class of all isometric representations on $L^p$-spaces contain harmonic cocycles because the group in question has property $(T)$. 

\begin{ques}\label{Question: Harmonic in every cohomology class}
    Let $G$ be a locally compact second countable group with property $(T)$. Do all continuous isometric representations of $G$ on $p$-uniformly convex spaces have a harmonic 1-cocycle in every first cohomology class?
\end{ques}

For $L^p$-spaces, Question \ref{Question: Harmonic in every cohomology class} holds since all isometric representations of property $(T)$ groups have spectral gap \cite[Theorem A]{BFGM}. Another family of spaces for which isometric representations of property $(T)$ groups have spectral gap are some non-commutative $L^p$-spaces \cite{olivier-kazhdan}. 

\begin{rem}
    For Koopman representations on non-commutative $L^p$-spaces coming from state-preserving actions on their respective von Neumann algebras, the first cohomology space is trivial when the group has property $(T)$ \cite[Theorem 5.5]{marrakchi-dlSalle} and hence these representations are not interesting for our purposes.
\end{rem} 

Another related question is to obtain upper bounds for $L^\infty$-growth. It is proven in \cite[Proposition 4]{lafforgue-typeneg} and in \cite[pages 207,208]{faraut_distances_1974} that for $G=\mathrm{SL}(2,\mathbb Q_p)$ and $G = \mathrm{SO}(n,1)$, for any unitary representation $\pi\colon G\to U(\Hi)$ and any $1$-cocycle $b$ in $Z^1(G,\pi)$, there exists $C = C(b) >0$ such that for every $n \geq 1$ we have:
\begin{equation*}
    \sup_{|g|_S\leq n}||b(g)||_\Hi \leq C \sqrt n.
\end{equation*}
It is then natural to ask the following:
\begin{ques}
    Let $2<p<\infty$. What are the upper bounds for the $L^\infty$-growth of cocycles with values in $L^p$-spaces of $\mathrm{SL}(2,\mathbb Q_p)$, $\mathrm{SO}(n,1)$, $\mathrm{SU}(n,1)$, $\mathrm{Sp}(n,1)$? 
\end{ques}
The proof of these bounds for unitary representations relies on the use of the Bochner-Godement theorem (cf. \cite[Théorème 3.1]{faraut_distances_1974}), so it is unclear how to extend this to $L^p$-spaces for $p \neq 2$.


\appendix

\section{Appendix}\label{Section: Appendix}

The goal of this appendix is to prove Proposition \ref{separable Lp space inside E_U}, which we state again now, this time in terms of cocycles.

\begin{prop}\label{Appendix: separable Lp space inside E_U}
    Let $G$ be a locally compact second countable group, $\pi\colon G \to \O(E)$ a continuous isometric representation of $G$ on some $L^p$-space $E$ and $b \in Z^1(G, \pi)$. There exists some closed separable $L^p$-space $F \subseteq E$ that contains $b(G)$ and is $\pi(G)$-invariant.
\end{prop}

The proof of this result was communicated to us by Mikael de la Salle. We thank him for his help and for letting us write down his proof.
We first prove Proposition \ref{Appendix: separable Lp space inside E_U} in a particular case.

\begin{prop}\label{Appendix: separable Lp space inside E_U in particular case}
    Let $\G$ be a countable group, $\pi \colon \G \to \O(E)$ be an isometric representation of $\G$ on some $L^p$-space $E = L^p (\Omega, \mathcal{A}, \mu)$, where $\mu$ is a probability measure, and $b \in Z^1(\G, \pi)$. There exists some closed separable $L^p$-space $F \subseteq E$ that contains $b(\G)$ and is $\pi(\G)$-invariant.
\end{prop}

\begin{proof}
    The Banach-Lamperti theorem \cite[3.1]{lamperti} says that $\pi$ can be expressed as $(\pi(g) f)(x) = h_g(x)f( T_g^{-1}(x))$ for $f \in E$, where $T_g : X \to X$ are measurable isomorphisms preserving the measure class of $\mu$, so that $\G \times X \to X, (g, x) \mapsto T_g(x)$ is a measurable action of $\G$ on $X$ and $h_g : \Omega \to \mathbb{R}$ is $\mathcal{A}$-measurable and satisfies $|h_g(x)|^p = \frac{\mathrm{d} ((T_g)_*\mu)}{\mathrm{d}\mu}(x)$ for $\mu$-almost every $x \in X$.

    The space $\mathrm{span}_\re \{ b(g), g \in \G \}$ is a separable (as $\G$ is countable) $\pi(\G)$-invariant (thanks to the cocycle relation) subspace of $E$. Since $\mu$ is a probability, the constant function $\mathbf{1}$ equal to 1 lives in $E$ and $\pi(g) \mathbf{1} = h_g$ for $g \in \G$. This implies that the space $E':= \mathrm{span}_\re \{ b(g), h_g \, | \,  g \in \G \}$ is still separable and $\pi(\G)$-invariant. 
    
    Let $\mathcal{B}$ the the smallest $\s$-algebra on $\Omega$ making all the functions in $E'$ measurable. In particular, the functions $h_g$ are $\mathcal{B}$-measurable and since for $f \in E'$ we have that $\pi(g)f$ is also $\mathcal{B}$-measurable, then $f \circ T_g$ is also $\mathcal{B}$-measurable. Hence the $\s$-algebra $\mathcal{B}$ is preserved by $T_g$ for $g \in \G$. Therefore the $L^p$-space $F := L^p(\Omega, \mathcal{B}, \mu) \subseteq E$ is invariant by $T_g$ for $g \in \G$ and hence $\pi(\G)$-invariant. The space $F$ is closed as conditional expectation gives a bounded projection $E \to F$. Since $\G$ is countable, the $\s$-algebra $\mathcal{B}$ is countably generated and hence the $L^p$-space $F$ is separable.
\end{proof}

The proof of Proposition \ref{Appendix: separable Lp space inside E_U} now consists in being able to reduce to the hypotheses of Proposition \ref{Appendix: separable Lp space inside E_U in particular case}. Before this, we make one more reduction.

\begin{lem}\label{Appendix: Lemma}
    Let $\G$ be a countable group, $\pi\colon \G \to \O(E)$ be a continuous isometric representation of $\G$ on some $L^p$-space $ E = L^p(\Omega,\mathcal{A},\mu)$, where $1<p<\infty$ (here $(\Omega, \mathcal{A}, \mu)$ can be any measure space). Let $X$ be a countable $\pi(\G)$-invariant subset of $E$ such that $X \neq \{ 0 \}$. Then the subset $\Omega' = \bigcup_{f \in X} \Omega \setminus f^{-1}(\{ 0 \} )$ is $\mathcal{A}$-measurable, the measure $\mu$ restricted to $\Omega'$ is $\s$-finite and $E' = L^p(\Omega', \mathcal{A', \mu})$ is $\pi(\G)$-invariant, where $\mathcal{A'}$ denotes the $\s$-algebra of restrictions of elements in $\mathcal{A}$ to $\Omega'$.
\end{lem}

\begin{proof}[Proof of Lemma]
    $\Omega' $ is measurable as a countable union of measurable subsets. Since $p < \infty$, for every $f \in E$, we have $\mu(\Omega \setminus f^{-1}([-\frac{1}{n}, \frac{1}{n}])) < \infty$. For $f \in E$, we have $\Omega \setminus f^{-1}(\{ 0 \} ) = \bigcup_{n>0}\Omega \setminus f^{-1}([-\frac{1}{n}, \frac{1}{n}])$, so $\Omega \setminus f^{-1}(\{ 0 \} )$ is $\s$-finite and $\Omega'$ is $\s$-finite as a countable union of $\s$-finite sets.

    It remains to show that $E'$ is $\pi(\G)$-invariant. In order to do this, we will first show that the complementary subspace $E'' = L^p(\Omega'', \mathcal{A}'', \mu)$ is $\pi(\G)$-invariant, where $\Omega'' = \Omega \setminus \Omega' = \bigcap_{f \in X} f^{-1}(\{ 0 \} )$ and $\mathcal{A}''$ the $\s$-algebra consisting of restrictions of elements of $\mathcal{A}$ to $\Omega''$. Let $h \in E''$. For any $f \in X$ we have: 
    \begin{equation*}
        ||f + h ||_p^p = ||f||_p^p + ||h||_p^p.
    \end{equation*}
    Since $\G$ acts by isometries and $X$ is $\pi(\G)$-invariant, we have for $g \in \G$ and $f \in X$:
\begin{equation*}
    ||f + \pi(g) h ||_p^p = ||f||_p^p + ||\pi(g) h||_p^p.
\end{equation*}
Hence \cite[2.1]{lamperti} implies that $\pi(g)h$ is supported on $f^{-1}(\{ 0 \} )$ for any $f \in X$ and hence on $\Omega''$. This shows that $\pi(g) h \in E''$, so the space $E''$ is $\pi(\G)$-invariant.

Using that $E''$ is $\pi(\G)$-invariant, we now show that $E'$ is $\pi(\G)$-invariant in a similar way. Let $h \in E'$ and $g \in \G$. We write $\pi(g)h = x + y$ where $x \in E'$ and $y \in E''$. Since $\pi$ is isometric, we have:
\begin{equation*}
    ||x||_p^p = ||\pi(g)h-y||_p^p = ||h - \pi(g^{-1})y||_p^p.
\end{equation*}
But we know that $E''$ is $\pi(\G)$-invariant, hence $\pi(g^{-1})y\in E''$ and since $h \in E'$ we have:
\begin{equation*}
    ||h - \pi(g^{-1})y||_p^p = ||h||^p + ||\pi(g^{-1})y||_p^p = ||h||_p^p + ||y||_p^p.
\end{equation*}
Since $||h||_p^p = ||\pi(g) h||_p^p =||x||_p^p + ||y||_p^p$, we obtain that:
\begin{equation*}
    ||x||_p^p = ||x||_p^p + 2 ||y||_p^p.
\end{equation*}
Hence $y = 0$, which means that $\pi(g)h = x \in E'$. We have shown that $E'$ is $\pi(\G)$-invariant.
\end{proof}

\begin{proof}[Proof of Proposition \ref{Appendix: separable Lp space inside E_U}]
    Let $G$ be a locally compact second countable group. By second countability, there exists a countable dense subset $Z$ of $G$. Let $\G$ be the countable dense subgroup of $G$ generated by $Z$. If $\pi\colon \G \to \O(E)$ is a continuous isometric representation of $G$ on an $L^p$-space $E = L^p(\Omega, \mathcal{A}, \mu)$, the continuity of the action says that closed $\pi(\G)$-invariant subspaces are exactly closed $\pi(G)$-invariant subspaces. By Lemma \ref{Appendix: Lemma}, we can restrict to a subset $\Omega'$ of $\Omega$ so that the measure $\mu$ on $\Omega'$ is $\s$-finite and $L^p(\Omega', \mathcal{A}', \mu)$ is $\pi(\G)$-invariant. Since any $\s$-finite measure $\mu$ is in the measure class of a probability measure $\mu_1$, the spaces $ L^p(\Omega', \mathcal{A}', \mu)$ and  $L^p(\Omega', \mathcal{A}', \mu_1)$ are isometric and the space $L^p(\Omega', \mathcal{A}', \mu_1)$ is $\G$-invariant for the induced $\G$-action given by composing with the isometry. Now Proposition \ref{Appendix: separable Lp space inside E_U in particular case} says that there exists a separable $\G$-invariant subspace $F= L^p(\Omega', \mathcal{B}, \mu_1)$ of $L^p(\Omega', \mathcal{A}', \mu_1)$. The space $F$ inherits a $G$-action from the initial $G$-action on $E$. By density of $\G$, $F$ is $G$-invariant.
\end{proof}

\bibliographystyle{amsalpha}
\bibliography{refs.bib}

\noindent Antonio López Neumann \\
Université Paris Cité, Sorbonne Université, CNRS, IMJ-PRG, F-75013 Paris, France \\
lopezneumann@imj-prg.fr \\

\noindent Juan Paucar Zanabria\\
Sección de Matemáticas, Departamento de Ciencias, PUCP, Lima, Perú \\
jlpaucar@pucp.edu.pe\\
\end{document}